\documentclass[twoside,leqno,10pt, A4]{amsart}
\usepackage{amsfonts}
\usepackage{amsmath}
\usepackage{amscd}
\usepackage{amssymb}
\usepackage{amsthm}
\usepackage{amsrefs}
\usepackage{latexsym}
\usepackage{mathrsfs}
\usepackage{bbm}
\usepackage{amscd}
\usepackage{amssymb}
\usepackage{amsthm}
\usepackage{amsrefs}
\usepackage{latexsym}
\usepackage{mathrsfs}
\usepackage{bbm}
\usepackage{enumerate}
\usepackage{graphicx}
\allowdisplaybreaks
\usepackage{color}
\begin{document}

\newtheorem{theorem}[subsection]{Theorem}
\newtheorem{proposition}[subsection]{Proposition}
\newtheorem{lemma}[subsection]{Lemma}
\newtheorem{corollary}[subsection]{Corollary}
\newtheorem{conjecture}[subsection]{Conjecture}
\newtheorem{prop}[subsection]{Proposition}
\newtheorem{defin}[subsection]{Definition}

\numberwithin{equation}{section}
\newcommand{\mr}{\ensuremath{\mathbb R}}
\newcommand{\mc}{\ensuremath{\mathbb C}}
\newcommand{\N}{\mathbb{N}}
\newcommand{\dif}{\mathrm{d}}
\newcommand{\intz}{\mathbb{Z}}
\newcommand{\ratq}{\mathbb{Q}}
\newcommand{\natn}{\mathbb{N}}
\newcommand{\comc}{\mathbb{C}}
\newcommand{\rear}{\mathbb{R}}
\newcommand{\prip}{\mathbb{P}}
\newcommand{\uph}{\mathbb{H}}
\newcommand{\fief}{\mathbb{F}}
\newcommand{\majorarc}{\mathfrak{M}}
\newcommand{\minorarc}{\mathfrak{m}}
\newcommand{\sings}{\mathfrak{S}}
\newcommand{\fA}{\ensuremath{\mathfrak A}}
\newcommand{\mn}{\ensuremath{\mathbb N}}
\newcommand{\mq}{\ensuremath{\mathbb Q}}
\newcommand{\half}{\tfrac{1}{2}}
\newcommand{\f}{f\times \chi}
\newcommand{\summ}{\mathop{{\sum}^{\star}}}
\newcommand{\chiq}{\chi \bmod q}
\newcommand{\chidb}{\chi \bmod db}
\newcommand{\chid}{\chi \bmod d}
\newcommand{\sym}{\text{sym}^2}
\newcommand{\hhalf}{\tfrac{1}{2}}
\newcommand{\sumstar}{\sideset{}{^*}\sum}
\newcommand{\sumprime}{\sideset{}{'}\sum}
\newcommand{\sumprimeprime}{\sideset{}{''}\sum}
\newcommand{\sumflat}{\sideset{}{^\flat}\sum}
\newcommand{\shortmod}{\ensuremath{\negthickspace \negthickspace \negthickspace \pmod}}
\newcommand{\V}{V\left(\frac{nm}{q^2}\right)}
\newcommand{\sumi}{\mathop{{\sum}^{\dagger}}}
\newcommand{\mz}{\ensuremath{\mathbb Z}}
\newcommand{\leg}[2]{\left(\frac{#1}{#2}\right)}
\newcommand{\muK}{\mu_{\omega}}
\newcommand{\thalf}{\tfrac12}
\newcommand{\lp}{\left(}
\newcommand{\rp}{\right)}
\newcommand{\Lam}{\Lambda_{[i]}}
\newcommand{\lam}{\lambda}
\newcommand{\af}{\mathfrak{a}}
\newcommand{\sw}{S_{[i]}(X,Y;\Phi,\Psi)}
\newcommand{\lz}{\left(}
\newcommand{\pz}{\right)}
\newcommand{\bfrac}[2]{\lz\frac{#1}{#2}\pz}
\newcommand{\odd}{\mathrm{\ primary}}
\newcommand{\even}{\text{ even}}
\newcommand{\res}{\mathrm{Res}}
\newcommand{\sumn}{\sumstar_{(c,1+i)=1}  w\left( \frac {N(c)}X \right)}
\newcommand{\lab}{\left|}
\newcommand{\rab}{\right|}
\newcommand{\Go}{\Gamma_{o}}
\newcommand{\Ge}{\Gamma_{e}}
\newcommand{\M}{\widehat}
\def\su#1{\sum_{\substack{#1}}}
\newcommand{\Echar}{\mathbb{E}^{\text{char}}}
\newcommand{\E}{\mathbb{E}}
\newcommand{\p}{\mathbb{P}}

\theoremstyle{plain}
\newtheorem{conj}{Conjecture}
\newtheorem{remark}[subsection]{Remark}

\newcommand{\pfrac}[2]{\left(\frac{#1}{#2}\right)}
\newcommand{\pmfrac}[2]{\left(\mfrac{#1}{#2}\right)}
\newcommand{\ptfrac}[2]{\left(\tfrac{#1}{#2}\right)}
\newcommand{\pMatrix}[4]{\left(\begin{matrix}#1 & #2 \\ #3 & #4\end{matrix}\right)}
\newcommand{\ppMatrix}[4]{\left(\!\pMatrix{#1}{#2}{#3}{#4}\!\right)}
\renewcommand{\pmatrix}[4]{\left(\begin{smallmatrix}#1 & #2 \\ #3 & #4\end{smallmatrix}\right)}
\def\en{{\mathbf{\,e}}_n}

\newcommand{\ppmod}[1]{\hspace{-0.15cm}\pmod{#1}}
\newcommand{\ccom}[1]{{\color{red}{Chantal: #1}} }
\newcommand{\acom}[1]{{\color{blue}{Alia: #1}} }
\newcommand{\alexcom}[1]{{\color{green}{Alex: #1}} }
\newcommand{\hcom}[1]{{\color{brown}{Hua: #1}} }

\makeatletter
\def\widebreve{\mathpalette\wide@breve}
\def\wide@breve#1#2{\sbox\z@{$#1#2$}%
     \mathop{\vbox{\m@th\ialign{##\crcr
\kern0.08em\brevefill#1{0.8\wd\z@}\crcr\noalign{\nointerlineskip}%
                    $\hss#1#2\hss$\crcr}}}\limits}
\def\brevefill#1#2{$\m@th\sbox\tw@{$#1($}%
  \hss\resizebox{#2}{\wd\tw@}{\rotatebox[origin=c]{90}{\upshape(}}\hss$}
\makeatletter

\title[Moments of random multiplicative functions with Fourier coefficients of modular forms]{Low moments of random multiplicative functions twisted by Fourier coefficients of modular forms}

%%\date{\today}
\author[P. Gao]{Peng Gao}
\address{School of Mathematical Sciences, Beihang University, Beijing 100191, China}
\email{penggao@buaa.edu.cn}

\author[L. Zhao]{Liangyi Zhao}
\address{School of Mathematics and Statistics, University of New South Wales, Sydney NSW 2052, Australia}
\email{l.zhao@unsw.edu.au}

\begin{abstract}
  Let $\lambda(n)$ denote the Fourier coefficients of a fixed modular form and $h(n)$ a Steinhaus or Rademacher random multiplicative function. In this paper, we determine the order of magnitude of 
  \[ \E|\sum_{n \leq x} h(n)\lambda(n)|^{2q} \]
   for real $x$, $q$ with $0 \leq q \leq 1$.
\end{abstract}

\maketitle



\section{Introduction}\label{sec1}

The sums of arithemtic functions, such as character sums, are very important in number theory and hence have been studied extensively.  One may often model the behaviors of these sums using random multiplicative functions. For example, one may simulate a Dirichlet character (resp. the M\"obius function) by a Steinhaus (resp. Rademacher) random multiplicative function. Here we recall (see \cite{Harper20}) that for all $n \in \N$, a Steinhaus random multiplicative function $h(n)$ is defined to be $h(n) = \prod_{p^{a} || n} h(p)^{a}$, where we reserve the letter $p$ for a prime number throughout the paper, and $(h(p))_{p \; \text{prime}}$ is a sequence of independent random variables, each distributed uniformly on the complex unit circle.  A Rademacher random multiplicative function $h(n)$ is defined to be supported on square-free integers $n$ only such that $h(n) := \prod_{p |n} h(p)$, where $(h(p))_{p \; \text{prime}}$ is a sequence of independent random variables taking values $\pm 1$ with probability $1/2$ each. \newline

 When estimating sums of a sequence of oscillating arithmetic functions,  the well-known square-root cancellation heuristic suggests that they should typically be of size about the square-root, but not smaller, of the trivial bounds obtained by simply using the triangle inequality.  However, A. J. Harper (see \cites{Harper20, Harper23}) showed that there are sums, of both random multiplicative functions and Dirichlet characters, for which that the above heuristic does not apply.  More precisely, he determined the order of magnitude of the low moments of sums of Steinhaus and Rademacher random multiplicative functions in \cite{Harper20}, proving that for a Steinhaus random multiplicative function or a Rademacher random multiplicative function $h(n)$, one has uniformly for all large real $x$ and $0 \leq q \leq 1$,
\begin{align}
\label{randlowmoment}
 \E|\sum_{n \leq x} h(n)|^{2q} \asymp \left( \frac{x}{1 + (1-q)\sqrt{\log\log x}} \right)^q,
\end{align}
 where $\mathbb{E}$ denotes the expectation throughout the paper. \newline

  Further, Harper applied the ideas in \cite{Harper20} to establish in \cite[Theorem 3]{Harper23} that for large primes $Q$, uniformly for any $x \in [1,Q]$, $q \in [0,1]$, and multiplicative function $u(n)$ with modulus $1$ at primes and not exceeding $1$ at prime powers,
\begin{align}
\label{genlowmoment}
 \frac{1}{Q-1} \sum_{\chi \shortmod Q} \Big| \sum_{n \leq x}  \chi(n)u(n)\Big| ^{2q} \ll \Biggl(\frac{x}{1 + (1-q)\sqrt{\log\log(10L)}} \Biggr)^q, \quad \mbox{where} \; L=L_Q=  \min\{x,Q/x\}
\end{align}
and $\chi(n)$ runs over Dirichlet characters modulo $Q$. \newline

   Note that the special case of \eqref{randlowmoment} with $q=1/2$ implies that
\begin{equation}
\label{squarerootrand}
     \E|\sum_{n \leq x} h(n)| =o(\sqrt{x}).
\end{equation}
   Similarly, the special case of \eqref{genlowmoment} with $q=1/2, u(n)\equiv 1$ in \eqref{genlowmoment} implies that when both $x$ and $Q/x$ tend to infinity with $Q$,
\begin{equation}
\label{squareroot}
     \frac{1}{Q-1} \sum_{\chi \shortmod Q}  \Big| \sum_{n\leq x} \chi(n) \Big| =o(\sqrt{x}).
\end{equation}
Hence, \eqref{squarerootrand} and \eqref{squareroot} imply that the summands in $\sum_{n\leq x} h(n)$, $\sum_{n\leq x} \chi(n)$ typically neutralize one another more than the square-root cancellation heuristic suggests. \newline

  In light of \eqref{randlowmoment}, it is natural to consider bounds for more general sums $\sum_{n \leq x} h(n)a(n)$, where $\{a(n)\}$ is a complex valued sequence and $h(n)$ a random multiplicative function.  One such example has been studied in \cite{Xu24}. Here we note that, if $a(n)$ is multiplicative with $a(1)=1$, $|a(n)| \leq 1$, then a straightforward modification of the arguments in \cite{Harper20} may allow one to obtain bounds for $\E|\sum_{n \leq x} h(n)a(n)|^{2q}$ with $0 \leq q \leq 1$.  Thus we turn our attention to other types of sequences.  One candidate is to set $a(n)=\lambda(n)$, where $\lambda(n)$ equals the Fourier coefficient of a fixed holomorphic Hecke eigenform $f$ of weight $\kappa \equiv 0 \pmod 4$ for the full modular group $SL_2 (\mathbb{Z})$. Here we recall that the Fourier expansion of $f$ at infinity is given by
\[
f(z) = \sum_{n=1}^{\infty} \lambda (n) n^{(\kappa -1)/2} e(nz), \quad \mbox{where} \quad e(z) = \exp (2 \pi i z).
\]

In fact, low moments of the sum $\sum_{n \leq x} \chi(n)\lambda(n)$ has already been studied in \cite{G&Wu25-12}.  It is shown there that for large primes $Q$, any $1 \leq x \leq Q$ and $0 \leq q \leq 1$, we have, for $L$ defined in \eqref{genlowmoment},
\begin{align*}
%%\label{genlowmoment}
 \frac{1}{Q-1} \sum_{\chi \shortmod Q} \Big| \sum_{n \leq x}  \chi(n)\lambda(n)\Big| ^{2q} \ll \Biggl(\frac{x}{1 + (1-q)\sqrt{\log\log(10L)}} \Biggr)^q.
\end{align*}

The aim of this paper is to determine the order of magnitude of the low moments of sums of the form $\sum_{n \leq x} h(n)\lambda(n)$ for $h$ being  a Steinhaus or a Rademacher random multiplicative function. Our result is as follows.
\begin{theorem}
\label{lowerboundsfixedmodmean}
With the notation as above, suppose $h(n)$ is a Steinhaus random multiplicative function or a Rademacher random multiplicative function.  Then uniformly for all large $x$ and $0 \leq q \leq 1$, we have
\begin{align}
\label{mainestimation}
 \E|\sum_{n \leq x} h(n)\lambda(n)|^{2q} \asymp \left( \frac{x}{1 + (1-q)\sqrt{\log\log x}} \right)^q .
\end{align}
\end{theorem}

This work is motivated by the result of Harper \cite{Harper20}, which provides a general framework for the proof of Theorem \ref{lowerboundsfixedmodmean}.   Additionally, we need to utlize the arithmetic properties of $\lambda(n)$.

\section{Preliminaries}
\label{sec 2}

In this section, we gather some auxiliary results necessary in the proof of Theorem \ref{lowerboundsfixedmodmean}.

\subsection{Properties of $\lambda(n)$}
\label{sec:cusp form}

    Recall that $f$ is a fixed holomorphic Hecke eigenform $f$ of weight $\kappa \equiv 0 \pmod 4$ for the full modular group $SL_2 (\mathbb{Z})$. For $\Re(s)>1$, the associated modular $L$-function $L(s, f)$ is defined as
\begin{align*}
%%\label{Lphichi}
L(s, f ) &= \sum_{n=1}^{\infty} \frac{\lambda(n)}{n^s}
 = \prod_{p} \left(1 + \sum^{\infty}_{i=1}\frac{\lambda(p^i)}{p^{is}}\right)=\prod_{p} \left(1 - \frac{\alpha_p }{p^s} \right)^{-1}\left(1 - \frac{\beta_p }{p^s} \right)^{-1}.
\end{align*}
  It follows that $\lambda(n)$ is multiplicative and, for any prime $p$ and non-negative integer $m$,
\begin{align} \label{Lambdapkrel}
 \lambda(p^m)=\sum^{m}_{j=0}\alpha^{m-j}_p\beta^{j}_p.
\end{align}

Moreover, from Deligne's proof \cite{D} of the Weil conjecture,
\begin{align}
\label{alpha}
|\alpha_{p}|=|\beta_{p}|=1, \quad \alpha_{p}\beta_{p}=1.
\end{align}

Thus from \eqref{Lambdapkrel} and \eqref{alpha},
\begin{align} \label{alphalambdarel}
  \alpha_p+\beta_p=  \lambda(p), \quad \mbox{and} \quad \alpha^2_p+\beta^2_p= \lambda^2(p)-2=\lambda(p^2)-1.
\end{align}

 Furthermore, using \eqref{Lambdapkrel}, \eqref{alpha} and the multiplicativity of $\lambda(n)$, we deduce that $\lambda(n) \in \mr$, $\lambda (1) =1$ and for any $\varepsilon>0$,
\begin{align}
\label{lambdabound}
\begin{split}
  |\lambda(n)| \leq d(n) \ll n^{\varepsilon},
\end{split}
\end{align}
 where $d(n)$ is the number of positive divisors $n$ and can be bounded above as in \eqref{lambdabound} using \cite[Theorem 2.11]{MVa1}.

 Observe that $|\lambda(n)|^2=\lambda^2(n) \geq 0$ as $\lambda(n)$ is real. It was proved independently by R. A. Rankin \cite{Rankin39} and A. Selberg \cite{Selberg40} that, for $x \geq 1$,
\begin{align}
\label{lambdasquareasymp}
\begin{split}
  \sum_{n \leq x}\lambda^2(n) =\frac {L(1, \operatorname{sym}^2 f)}{\zeta(2)}x+O(x^{3/5}).
\end{split}
\end{align}
  Here $L(s, \operatorname{sym}^2 f)$ is the symmetric square $L$-function of $f$ defined for $\Re(s)>1$ by
 (see \cite[p. 137]{iwakow} and \cite[(25.73)]{iwakow})
\begin{align}
\label{Lsymexp}
\begin{split}
 L(s, \operatorname{sym}^2 f)=& \prod_p(1-\alpha^2_pp^{-s})^{-1}(1-p^{-s})^{-1}(1-\beta^2_pp^{-s})^{-1} \\
 = & \zeta(2s) \sum_{n \geq 1}\frac {\lambda(n^2)}{n^s}=\prod_{p}\left( 1-\frac {\lambda(p^2)}{p^s}+\frac {\lambda(p^2)}{p^{2s}}-\frac {1}{p^{3s}} \right)^{-1}.
\end{split}
\end{align}
 A result of G. Shimura \cite{Shimura} implies that the corresponding completed symmetric square $L$-function
\begin{align*}
%%\label{Lambdafdef}
 \Lambda(s, \operatorname{sym}^2 f)=& \pi^{-3s/2}\Gamma \Big( \frac {s+1}{2} \Big)\Gamma \Big(\frac {s+\kappa-1}{2}\Big) \Gamma \Big(\frac {s+\kappa}{2}\Big) L(s, \operatorname{sym}^2 f)
\end{align*}
  is entire and satisfies the functional equation $\Lambda(s, \operatorname{sym}^2 f)=\Lambda(1-s, \operatorname{sym}^2 f)$. Moreover, $L(s, \operatorname{sym}^2 f)$ has no pole at $s=1$. \newline

Next, we by include a result on certain sums over primes.
\begin{lemma}
\label{RS}
 Let $x \geq 2$. We have, for some constant $b_1$, $b_2$, $b_3$ with $b_2$, $b_3>0$,
\begin{align}
%%\label{merten}
%%\sum_{p\le x} \frac{1}{p} =& \log \log x + b_1+ O\Big(\frac{1}{\log x}\Big), \;  \\
\label{merten1}
\sum_{p\le x} \frac{\lambda^2(p)}{p} =& \log \log x + b_1+ O\Big(\frac{1}{\log x}\Big),  \\
\label{merten2-0}
\sum_{p\le x} \log p =&  x+O\Big(x \exp(-b_2\sqrt{\log x})\Big),  \; \mbox{and} \\
\label{merten2}
\sum_{p\le x} (\lambda^2(p)-1)\log p =&  O\Big(x \exp(-b_3\sqrt{\log x})\Big).
\end{align}
\end{lemma}
\begin{proof}
  The formula \eqref{merten1} follows from \cite[Lemma 2.1]{GHH}. The expression in \eqref{merten2-0} is given in \cite[Theorem 6.9]{MVa1}. Lastly, we have, for $\Re(s) \geq 2$,
\begin{align*}
%%\label{lambdasquareasymp}
\begin{split}
  -\frac {L'(s, \operatorname{sym}^2 f)}{L(s, \operatorname{sym}^2 f)}=\sum_{n \geq 1}\frac {\Lambda_{\operatorname{sym}^2 f}(n)}{n^s}.
\end{split}
\end{align*}
By \eqref{alphalambdarel} and \eqref{Lsymexp},
\begin{align*}
%%\label{lambdasquareasymp}
\begin{split}
  \Lambda_{\operatorname{sym}^2 f}(p)=(\alpha^2_p+\beta^2_p+1)\log p=(\lambda^2(p)-1)\log p.
\end{split}
\end{align*}
  Note further that $L(s, \operatorname{sym}^2 f)$ is a $\text{GL}(3)$ function.  By a result of W. D. Banks \cite{Banks97}, it does not admit any Siegel zero.  It follows from this and \cite[(5.49),(5.52)]{iwakow} that \eqref{merten2} holds. This completes the proof of the lemma.
\end{proof}

\subsection{Parseval’s identity}

    The following lemma is taken from \cite[Theorem 5.4]{MVa1}, which gives a version of Parseval’s identity for Dirichlet series.
\begin{lemma}
\label{parseval}
    Let $(a_n)_{n\geq 1}$ be a sequence of complex numbers and $F(s)=\sum_{n=1}^{\infty} a_nn^{-s}$ be the corresponding Dirichlet series. If $\sigma_c$ denotes its abscissa of convergence, then, for any $\sigma>\max\{ 0,\sigma_c \}$, we have
    \begin{equation*}
        \int\limits_{1}^{\infty} \frac{\big|\sum_{n\leq x}a_n\big|^2 }{x^{1+2\sigma }} \dif x=\frac{1}{2\pi}\int\limits_{-\infty}^{+\infty}\frac{|F(\sigma+it)|^2}{|\sigma+it|^2} \dif t.
    \end{equation*}
\end{lemma}

\subsection{Estimates on Expectations}

  Recall that we denote the expectation by $\mathbb{E}$.  Furthermore, denote the probability measure by $\p$ and the indicator function by $\textbf{1}$.  We first have a result dealing with Steinhaus random multiplicative functions.
\begin{lemma}
\label{eulerproduct}
    Let $h(n)$ be a Steinhaus random multiplicative function. We have, for any real $t$ and $u$, $400(1 + u^2) \leq x \leq y$, and $\sigma \geq - 1/\log y$,
\begin{align}
\label{Eest}
\begin{split}
 \E \prod_{x < p \leq y} & \Big|1-\frac{\alpha_ph(p)}{p^{1/2+\sigma}}\Big|^{-2}\Big|1-\frac{\beta_ph(p)}{p^{1/2+\sigma}}\Big|^{-2}
\Big|1-\frac{\alpha_ph(p)}{p^{1/2+\sigma+it}}\Big|^{-iu}\Big|1-\frac{\beta_ph(p)}{p^{1/2+\sigma+it}}\Big|^{-iu}\\
=&  \exp\Big(\sum_{x < p \leq y} \frac{\lambda^2(p)(1 + iu\cos(t\log p) - u^{2}/4)}{p^{1 + 2\sigma}} + T(u)\Big),
\end{split}
\end{align}
where $T(u) = T_{x,y,\sigma,t}(u)$ satisfies $T(u) \ll \frac{1 + |u|^3}{\sqrt{x} \log x}$ for any $u$, and $T'(u) \ll \frac{1}{\sqrt{x} \log x}$ for $|u| \leq 1$. \newline

  Moreover, for any real $t,u,v$, $400(1 + u^2 + v^2) \leq x \leq y$, and $\sigma \geq - 1/\log y$, we have
\begin{align}
\label{Eest2}
\begin{split}
\E \prod_{x < p \leq y} & \left|1 - \frac{\alpha_ph(p)}{p^{1/2+\sigma}}\right|^{-(2+iu)}\left|1 - \frac{\beta_ph(p)}{p^{1/2+\sigma}}\right|^{-(2+iu)} \left|1 - \frac{\alpha_ph(p)}{p^{1/2+\sigma+it}}\right|^{-(2+iv)}\left|1 - \frac{\beta_ph(p)}{p^{1/2+\sigma+it}}\right|^{-(2+iv)}  \\
 = & \exp\Big ( \sum_{x < p \leq y} \frac{\lambda^2(p)((1+iu/2)^2 + (1+iv/2)^2)}{p^{1 + 2\sigma}} + \sum_{x < p \leq y} \frac{\lambda^2(p)(2+iu)(2+iv)\cos(t\log p)}{2p^{1+2\sigma}} + T(u,v)\Big ) ,
\end{split}
\end{align}
where $T(u,v) = T_{x,y,\sigma,t}(u,v)$ and obeys the bounds $T(u,v) \ll \frac{1 + |u|^3 + |v|^3}{\sqrt{x} \log x}$, $\frac{\partial T(u,v)}{\partial u} \ll \frac{1 + u^2 + v^2}{\sqrt{x} \log x}$, $\frac{\partial T(u,v)}{\partial v} \ll \frac{1 + u^2 + v^2}{\sqrt{x} \log x}$, and $\frac{\partial T(u,v)}{\partial u \partial v} \ll \frac{1 + |u| + |v|}{\sqrt{x} \log x}$.
\end{lemma}
\begin{proof}
 The proof for \eqref{Eest} is given in \cite[Lemma 2.9]{G&Wu25-12}. The relation in \eqref{Eest2} can be similarly obtained by extending the arguments used in the proof of \cite[Lemma 2.9]{G&Wu25-12} and that of \cite[Lemma 6]{Harper20}.
\end{proof}

 Note that as a special case of \eqref{Eest} and by a similar argument, we have for any $400 \leq x \leq y$, $\sigma \geq -1/\log y$ and any Steinhaus random multiplicative function $h(n)$,
\begin{align}
\label{Eprodsquare}
\begin{split}
\E \prod_{x < p \leq y} \Big|1-\frac{\alpha_ph(p)}{p^{1/2+\sigma}}\Big|^{-2}\Big|1-\frac{\beta_ph(p)}{p^{1/2+\sigma}}\Big|^{-2} = & \exp\Big (\sum_{x < p \leq y} \frac{\lambda^2(p)}{p^{1 + 2\sigma}} + O\Big(\frac{1}{\sqrt{x}\log x} \Big) \Big ),  \quad \mbox{and} \\
\E \prod_{x < p \leq y} \Big|1-\frac{\alpha_ph(p)}{p^{1/2+\sigma}}\Big|^{2}\Big|1-\frac{\beta_ph(p)}{p^{1/2+\sigma}}\Big|^{2} = & \exp\Big (\sum_{x < p \leq y} \frac{\lambda^2(p)}{p^{1 + 2\sigma}} + O\Big(\frac{1}{\sqrt{x}\log x} \Big) \Big ).
\end{split}
\end{align}

Similarly, \eqref{Eest2} implies that for any real $t$, $400 \leq x \leq y$, $\sigma \geq -1/\log y$ and Steinhaus random multiplicative function $h(n)$,
\begin{align}
\label{sttwoprods}
\begin{split}
\E \prod_{x < p \leq y} & \left|1 - \frac{\alpha_ph(p)}{p^{1/2+\sigma}}\right|^{-2}\left|1 - \frac{\beta_ph(p)}{p^{1/2+\sigma}}\right|^{-2} \left|1 - \frac{\alpha_ph(p)}{p^{1/2+\sigma}}\right|^{-2}\left|1 - \frac{\beta_ph(p)}{p^{1/2+\sigma}}\right|^{-2} \\
& =  \exp\Big ( \sum_{x < p \leq y} \frac{\lambda^2(p)(2 + 2\cos(t\log p))}{p^{1+2\sigma}} + O\Big(\frac{1}{\sqrt{x}\log x} \Big) \Big ) .
\end{split}
\end{align}

  Our next result gives a Rademacher analogue of Lemma \ref{eulerproduct}.
\begin{lemma}
\label{eulerproductR}
Let $h(n)$ be a Rademacher random multiplicative function, then for any real $t_1 , t_2$ and $u$, $400(1 + u^2) \leq x \leq y$, and $\sigma \geq - 1/\log y$, we have
\begin{align*}
%%\label{Eprodsquare}
\begin{split}
\E \prod_{x < p \leq y} & \left|1 + \frac{\lambda(p)h(p)}{p^{1/2+\sigma+it_1}}\right|^{2}\left|1 + \frac{\lambda(p) h(p)}{p^{1/2+\sigma+i(t_1 + t_2)}}\right|^{iu} \\
 = & \exp\Big(\sum_{x < p \leq y} \frac{\lambda^2(p)(1 + iuc(t_1 , t_2 , p) - (u^{2}/4)(1+\cos(2(t_1 + t_2)\log p)))}{p^{1 + 2\sigma}} + \tilde T(u)\Big) .
\end{split}
\end{align*}
Here $c(t_1 , t_2 , p) = 2\cos(t_1 \log p) \cos((t_1 + t_2)\log p) - \tfrac12 \cos(2(t_1+t_2)\log p)$, and $\tilde T(u) = \tilde T_{x,y,\sigma,t_1,t_2}(u)$ satisfies $\tilde T_(u) \ll \frac{1 + |u|^3}{\sqrt{x} \log x}$ for any $u$, and $\tilde T'(u) \ll \frac{1}{\sqrt{x} \log x}$ when $|u| \leq 1$.
\end{lemma}
\begin{proof}
 Our proof is a variant of those of \cite[Lemmas 1--2]{Harper20}.  Set
$$R_{p}(t) := \Re\log \Big(1 + \frac{\lambda(p)h(p)}{p^{1/2+\sigma+it}} \Big).$$
Then
\begin{align*}
%%\label{Eprodsquare}
\begin{split}
& \left|1 + \frac{\lambda(p)h(p)}{p^{1/2+\sigma+it_1}}\right|^{2}\left|1 + \frac{\lambda(p) h(p)}{p^{1/2+\sigma+i(t_1 + t_2)}}\right|^{iu}
= \exp\big (2 R_{p}(t_1) +iu R_{p}(t_1 + t_2) \big )  =  1 + \sum_{j=1}^{\infty} \frac{(2 R_{p}(t_1) +iu R_{p}(t_1 + t_2))^j}{j!} .
\end{split}
\end{align*}

Applying Taylor expansion, the estimation given in \eqref{lambdabound} and the fact that $h(p) \in \{\pm 1\}$, we see that
$$R_{p}(t) = \sum_{k=1}^{\infty} (-1)^{k-1} \frac{\Re(\lambda(p)h(p)p^{-it})^k}{k p^{k(1/2+\sigma)}} = \frac{\lambda(p)h(p) \cos(t\log p)}{p^{1/2+\sigma}} - \frac{\lambda^2(p)\cos(2t\log p)}{2p^{1+2\sigma}} + O\Big(\frac{1}{p^{3/2+3\sigma}}\Big).$$
  The above implies that $\E R_{p}(t) = - \frac{\lambda^2(p)\cos(2t\log p)}{2p^{1+2\sigma}} + O\Big(\frac{1}{p^{3/2+3\sigma}}\Big)$. Note also
\begin{align*}
%%\label{Eprodsquare}
\begin{split}
 \E R_{p}(t)^2 =& \frac{\lambda^2(p)\cos^{2}(t\log p)}{p^{1 + 2\sigma}} + O\Big(\frac{1}{p^{3/2+3\sigma}}\Big) = \frac{\lambda^2(p)(1+\cos(2t\log p))}{2 p^{1 + 2\sigma}} + O\Big(\frac{1}{p^{3/2+3\sigma}}\Big) ,  \\
 \E R_{p}(t_1) R_{p}(t_1 + t_2) =& \frac{\lambda^2(p)\cos(t_1 \log p) \cos((t_1 + t_2)\log p)}{p^{1 + 2\sigma}} + O\Big(\frac{1}{p^{3/2+3\sigma}}\Big).
\end{split}
\end{align*}
  Moreover, we apply \eqref{lambdabound} to see that, for $j \geq 3$, $|R_{p}(t)^j| \leq (\sum_{k=1}^{\infty} \frac{|\lambda(p)|^k}{p^{k(1/2+\sigma)}})^j \leq (\sum_{k=1}^{\infty} \frac{2^k}{p^{k(1/2+\sigma)}})^j= \frac{2^j}{(p^{1/2+\sigma}-2)^j}$. \newline

  We observe from the proof of \cite[Lemma 1]{Harper20} that for primes $y \geq p > x \geq 400(1+u^2)$,  we have $(2+|u|)/p^{1/2 + \sigma} \leq e/5$. It follows that for such primes,
\begin{align*}
%%\label{Eprodsquare}
\begin{split}
 \E & \left|1 + \frac{\lambda(p)h(p)}{p^{1/2+\sigma+it_1}}\right|^{2} \left|1 + \frac{\lambda(p)h(p)}{p^{1/2+\sigma+i(t_1 + t_2)}}\right|^{iu}  \\
& =  1 - \frac{\lambda^2(p)\cos(2t_1 \log p)}{p^{1+2\sigma}} - \frac{\lambda^2(p)(iu/2)\cos(2(t_1 + t_2)\log p)}{p^{1+2\sigma}} + O(\frac{1 + |u|}{p^{3/2+3\sigma}})  \\
& \hspace*{1cm} + \frac{\lambda^2(p)(1+\cos(2t_1 \log p)) + 2iu \lambda^2(p)\cos(t_1 \log p) \cos((t_1 + t_2)\log p) - \frac{u^{2}}{4}\lambda^2(p)(1+\cos(2(t_1 + t_2)\log p))}{p^{1+2\sigma}}  \\
& \hspace*{1cm}  + \E\sum_{j=3}^{\infty} \frac{(2 R_{p}(t_1) +iu R_{p}(t_1 + t_2))^j}{j!}  + O\Big( \frac{1 + |u|}{p^{3/2+3\sigma}} + \frac{1 + u^2}{p^{3/2+3\sigma}} \Big) \\
& =  1 + \lambda^2(p)\frac{1+iuc(t_1 , t_2 , p) - (u^{2}/4)(1+\cos(2(t_1 + t_2)\log p))}{p^{1+2\sigma}} + D_{p}(u) ,
\end{split}
\end{align*}
where $D_{p}(u)$ satisfies $D_{p}(u) \ll \frac{1+|u|^3}{p^{3/2+3\sigma}} \ll \frac{1+|u|^3}{p^{3/2}}$, and $D_{p}'(u) \ll \frac{1}{p^{3/2+3\sigma}} \ll \frac{1}{p^{3/2}}$ for $|u| \leq 1$. \newline

We then proceed as in the proof of \cite[Lemma 1]{Harper20} to see that the assertions of the lemma are valid. This completes the proof of the lemma.
\end{proof}

  Analogous to \eqref{Eprodsquare}, we deduce from (the proof) of Lemma \ref{eulerproductR} that for any $400 \leq x \leq y$, $\sigma \geq -1/\log y$, $t_1 \in \rear$ and Rademacher random multiplicative $h(n)$,
\begin{align}
\label{EprodsquareR}
\begin{split}
\E \prod_{x < p \leq y} \left|1 + \frac{\lambda(p)h(p)}{p^{1/2+\sigma+it_1}}\right|^{2} =& \exp\Big (\sum_{x < p \leq y} \frac{ \lambda^2(p)}{p^{1 + 2\sigma}} + O\Big(\frac{1}{\sqrt{x}\log x} \Big) \Big ) , \quad \mbox{and}  \\
\E \prod_{x < p \leq y} \left|1 + \frac{\lambda(p)h(p)}{p^{1/2+\sigma+it_1}}\right|^{-2} =& \exp\Big (\sum_{x < p \leq y} \frac{\lambda^2(p)(1 + 2\cos(2 t_1 \log p))}{p^{1 + 2\sigma}} + O\Big(\frac{1}{\sqrt{x}\log x}\Big) \Big ) .
\end{split}
\end{align}

  Our next result is an extension of the lower bound part of Khintchine's inequality (see \cite[Lemma 8.1, Chap. 3]{gut}) to both the Steinhaus and the Rademacher cases.
\begin{lemma}
\label{Khintchine}
     Let $h(n)$ be either a Steinhaus or a Rademacher  random multiplicative function and $(a_n)_{n\geq 1}$ a sequence of complex numbers. The we have, for any $m>1$,
\begin{align}
\label{KTlower}
\begin{split}
 \E(|\sum_n h(n)\lambda(n)a_n|^m)^{1/m}
\gg (\sum_n \lambda^2(n)|a_n|^2)^{1/2}.
\end{split}
\end{align}
\end{lemma}
\begin{proof}
For any $t>0$,
\begin{align*}
%%\label{EPnsmall}
 \E (e^{t\sum h(n)\lambda(n)a_n})=\prod_n \E (e^{t h(n)\lambda(n)a_n})=
\begin{cases}
\displaystyle \prod_n1, & \quad \text{if $h$ is Steinhaus}, \\
\displaystyle \prod_n \frac 12(e^{t \lambda(n)a_n}+e^{-t \lambda(n)a_n}), & \quad \text{if $h$ is Rademacher}.
\end{cases}
\end{align*}
 In both cases, the inequality $|\frac 12(e^{x}+e^{-x})| \leq e^{|x|^2}$ yields that
\begin{align*}
%%\label{EPnsmall}
 \E (e^{t\sum h(n)\lambda(n)a_n})\leq  \exp \left( \frac {t^2}2\displaystyle \sum_n \lambda^2(n)|a_n|^2 \right).
\end{align*}

  We now apply \cite[Lemma 1.1, Chap. 3]{gut} to derive that for any $t>0, \mu>0$,
\begin{align*}
%%\label{EPnsmall}
 \p (|\sum_n h(n)\lambda(n)a_n| \geq \mu) \leq \exp \left( -t\mu+\frac {t^2}2\displaystyle \sum_n \lambda^2(n)|a_n|^2 \right).
\end{align*}
  We set $t=\mu/\sum_n \lambda^2(n)|a_n|^2$ in the inequality above to arrive at
\begin{align*}
%%\label{EPnsmall}
 \p (|\sum_n h(n)\lambda(n)a_n| \geq \mu) \leq  \exp \left( \frac {-\mu^2}{2\sum_n \lambda^2(n)|a_n|^2} \right).
\end{align*}
  Now observe that for $m>1$,
\begin{align*}
%%\label{EPnsmall}
 \E(|f|^m)=m\int \mu^{m-1}\p(|f| \geq \mu) \dif \mu.
\end{align*}
  We then deduce that
\begin{align*}
%%\label{EPnsmall}
 \E(|\sum_n h(n)\lambda(n)a_n|^m) \leq m\int \mu^{m-1} \exp \left(\frac {-\mu^2}{2\sum_n \lambda^2(n)|a_n|^2} \right) \dif \mu=2^{m/2-1}m\Gamma \Big(\frac m2 \Big) \Big(\sum_n \lambda^2(n)|a_n|^2 \Big)^{m/2},
\end{align*}
  where $\Gamma(x)$ is the Gamma function. \newline

   Next note that for $m, m'>1$ satisfying $1/m+1/m'=1$, we have
\begin{align*}
%%\label{EPnsmall}
\begin{split}
 \Big(\sum_n \lambda^2(n)|a_n|^2\Big)^{2}=\E(|\sum_n h(n)\lambda(n)a_n|^2) \leq & \E(|\sum_n h(n)\lambda(n)a_n|^m)^{1/m} \E(|\sum_n h(n)\lambda(n)a_n|^{m'})^{1/m'} \\
\ll & \Big(\sum_n \lambda^2(n)|a_n|^2\Big)^{1/2}\E(|\sum_n h(n)\lambda(n)a_n|^m)^{1/m}.
\end{split}
\end{align*}
The above renders valid the bound in \eqref{KTlower} and completes the proof of the lemma.
\end{proof}

\subsection{Some Probability Results}
\label{secprobcalc}

  In this section, we collect some Girsanov-type results on various probabilities concerning both Steinhaus and  Rademacher random multiplicative functions. Our treatments in this section largely follow those in \cite[Sections 3.2--3.3]{Harper20}. Let $x$ be large and $-1/100 \leq \sigma \leq 1/100$.  For $x \in \rear$, set $\lfloor x \rfloor = \max \{ n \in \intz : n \leq x \}$ and $\lceil x \rceil = \min \{ n \in \intz : n \geq x \}$.  Let $(l_j)_{j=1}^{n}$ be a strictly decreasing sequence of non-negative integers such that $l_1 \leq \log\log x - 2$, and $(x_j)_{j=1}^{n}$ a corresponding increasing sequence of real numbers such that $x_j := x^{e^{-(l_j + 1)}}$. \newline

  For a Steinhaus random multiplicative function $h(n)$, we define, as in \cite[Section 3.2]{Harper20}, a new probability measure $\tilde{\p} = \tilde{\p}_{f,x,\sigma}$  such that for each event $A$,
$$ \tilde{\p}(A) := \frac{\E \textbf{1}_{A} \prod_{p \leq x^{1/e}} \left|1 - \frac{\alpha_p h(p)}{p^{1/2+\sigma}}\right|^{-2}\left|1 - \frac{\beta_p h(p)}{p^{1/2+\sigma}}\right|^{-2}}{\E \prod_{p \leq x^{1/e}} \left|1 - \frac{\alpha_p h(p)}{p^{1/2+\sigma}}\right|^{-2}\left|1 - \frac{\beta_p h(p)}{p^{1/2+\sigma}} \right|^{-2}}.$$
We also denote by $\tilde{\E}$ the expectation with respect to the measure $\tilde{\p}$ such that for any random variable $F$,
$$ \tilde{\E}F := \frac{\E F \prod_{p \leq x^{1/e}} \left|1 - \frac{\alpha_p h(p)}{p^{1/2+\sigma}}\right|^{-2}\left|1 - \frac{\beta_p h(p)}{p^{1/2+\sigma}}\right|^{-2}}{\E \prod_{p \leq x^{1/e}} \left|1 - \frac{\alpha_p h(p)}{p^{1/2+\sigma}}\right|^{-2}\left|1 - \frac{\beta_p h(p)}{p^{1/2+\sigma}} \right|^{-2}}. $$

  Similarly, for each $t \in \mr$ and any Rademacher random multiplicative function $h(n)$, we define a new probability measure $\tilde{\p}_{t}^{\text{Rad}} = \tilde{\p}_{f,x,\sigma,t}^{\text{Rad}}$ as in \cite[Section 3.3]{Harper20} such that for each event $A$,
$$ \tilde{\p}_{t}^{\text{Rad}}(A) := \frac{\E \textbf{1}_{A} \prod_{p \leq x^{1/e}} \left|1 + \frac{\lambda(p)h(p)}{p^{1/2+\sigma+it}}\right|^{2}}{\E \prod_{p \leq x^{1/e}} \left|1 + \frac{\lambda(p)h(p)}{p^{1/2+\sigma+it}}\right|^{2}}. $$
 Analogous to $\tilde{\E}$, we write $\tilde{\E}_{t}^{\text{Rad}}$ for the expectation with respect to the measure $\tilde{\p}_{t}^{\text{Rad}}$. \newline

As pointed out in \cite[Section 3.2]{Harper20}, the exact choice of the range of $p$ in the definition of $\tilde{\p}$ or $\tilde{\p}_{t}^{\text{Rad}}$ is not too important, as the independence of the $h(p)$ means that if an event $A$ does not involve a particular prime, then the expectation of that term will cancel out. \newline

We further set, for each $l \in \N \cup \{0\}$,
\begin{align}
\label{Ildef}
I_l(s) := \begin{cases}
\displaystyle \prod_{x^{e^{-(l+2)}} < p \leq x^{e^{-(l+1)}}} \Big(1 - \frac{\alpha_ph(p)}{p^s} \Big)^{-1} \Big(1 - \frac{\beta_ph(p)}{p^s} \Big)^{-1}, & \quad \text{$h$ is a Steinhaus random multiplicative function}, \\
\displaystyle \prod_{x^{e^{-(l+2)}} < p \leq x^{e^{-(l+1)}}} \Big(1 + \frac{\lambda(p)h(p)}{p^s}\Big), & \quad \text{$h$ is a Rademacher random multiplicative function}.
\end{cases}
\end{align}

The following result for the Steinhaus case is analogous to \cite[Lemma 3]{Harper20} and is established in \cite[Lemma 2.12]{G&Wu25-12}.
\begin{lemma}
\label{lem3}
With the notation as above, let $h(n)$ be a Steinhaus random multiplicative function. Suppose that $x_1$ is sufficiently large and $|\sigma| \leq 1/\log x_n$ and that $(v_j)_{j=1}^{n}$ is any sequence of real numbers satisfying for any $1 \leq j \leq n$,
 $$ |v_j| \leq (1/40)\sqrt{\log x_j} + 2. $$

Then, for any sequence of real numbers $(t_j)_{j=1}^{n}$,
\begin{align*}
%%\label{lambdansquarenlargesimplified}
& \tilde{\p}(v_j \leq \log|I_{l_j}(1/2 + \sigma + it_j)| \leq v_j + 1/j^2, \; \forall \; 1 \leq j \leq n)  =  \left(1+O\left(\frac{1}{x_{1}^{1/100}}\right) \right) \p(v_j \leq N_j \leq v_j + 1/j^2, \; \forall \; 1 \leq j \leq n) ,
\end{align*}
where the $N_j$'s are independent Gaussian random variables with mean and variance, respectively, 
\[ \sum_{x_{j}^{1/e} < p \leq x_j} \frac{\lambda^2(p)\cos(t_j \log p)}{p^{1 + 2\sigma}} \quad \mbox{and} \quad \sum_{x_{j}^{1/e} < p \leq x_j} \frac{\lambda^2(p)}{2p^{1 + 2\sigma}}. \]
\end{lemma}

Now Lemma \ref{lem3} gives rise to the following analogue of \cite[Lemma 4]{Harper20} which is proved in \cite[Lemma 2.13]{G&Wu25-12}.
\begin{lemma}
\label{lem4}
 With the notation as in Lemma \ref{lem3}, let $h(n)$ be a Steinhaus random multiplicative function. Suppose $(u_j)_{j=1}^{n}$ and $(v_j)_{j=1}^{n}$ are sequences of real numbers such that for any $1 \leq j \leq n$,
$$ -(1/80)\sqrt{\log x_j} \leq u_j \leq v_j \leq (1/80)\sqrt{\log x_j}. $$
Then we have
\begin{align*}
%%\label{lambdansquarenlargesimplified}
\p(u_j + 2 \leq \sum_{m=1}^{j} N_m \leq v_j - 2, \; \forall \; j \leq n) 
 \ll & \ \tilde{\p}(u_j \leq \sum_{m=1}^{j} \log|I_{l_m}(\tfrac{1}
{2} + \sigma + it_m)| \leq v_j, \; \forall \; j \leq n) \\
\ll & \ \p(u_j - 2 \leq \sum_{m=1}^{j} N_m \leq v_j + 2, \; \forall \; j \leq n),
\end{align*}
 where the Gaussian random variables $N_m$ are also as in Lemma \ref{lem3}. \newline

 Moreover, if the numbers $(t_j)_{j=1}^{n}$ satisfy $|t_j| \leq \frac{1}{j^{2/3} \log x_j}$,  then we have
\begin{align*}
 \p \Big((u_j - j) + O(1)  \leq \sum_{m=1}^{j} G_m \leq (v_j-j) - O(1), \; \forall \; j  & \leq n \Big)  \ll  \tilde{\p} \Big(u_j \leq \sum_{m=1}^{j} \log|I_{l_m}(1/2 + \sigma + it_m)| \leq v_j \; \forall \; j \leq n \Big) \\
& \ll  \p((u_j - j) - O(1) \leq \sum_{m=1}^{j} G_m \leq (v_j-j) + O(1), \; \forall \; j \leq n) , 
\end{align*}
where the $G_m$'s are independent Gaussian random variables, each having mean 0 and variance $\sum_{x_{m}^{1/e} < p \leq x_m} \frac{\lambda^2(p)}{2p^{1 + 2\sigma}}$.
\end{lemma}

  Now, using Lemmas \ref{lem3}--\ref{lem4} together with Probability Results 1 and 2 in \cite{Harper20}, we arrive at the following result analogue of \cite[Proposition 5]{Harper20}, which is established in  \cite[Proposition 2.14]{G&Wu25-12}.
\begin{prop}
\label{prop5}
With the notation as above, there exists a large natural number $B$ such that the following is valid. Let $(l_j)_{j=1}^{n}$ be a decreasing sequence of non-negative integers defined by $l_j := \lfloor \log\log x \rfloor - (B+1) - j$, where $n \leq \log\log x - (B+1)$ is large. Suppose that $|\sigma| \leq \frac{1}{e^{B+n+1}}$, and that $(t_j)_{j=1}^{n}$ is a sequence of real numbers satisfying $|t_j| \leq \frac{1}{j^{2/3} e^{B+j+1}}$ for all $j$. \newline

For any Steinhaus random multiplicative function $h$, let $I_{l}(s)$ be defined as in \eqref{Ildef}. Then uniformly for any large $a$ and function $g(j)$ satisfying $|g(j)| \leq 10\log j$,
$$ \tilde{\p}(-a -Bj \leq \sum_{m=1}^{j} \log|I_{l_m}(1/2 + \sigma + it_m)| \leq a + j + g(j), \; \forall \; j \leq n) \asymp \min\{1,\frac{a}{\sqrt{n}}\} . $$
\end{prop}

We now consider a variant of \eqref{Eest}. Let $h$ be a Steinhaus random multiplicative function. We define for each $t \in \mr$, large real $x$ and $-1/100 \leq \sigma \leq 1/100$, a probability measure $\tilde{\p}_{t}^{\text{St}, (2)} = \tilde{\p}_{x,\sigma,t}^{\text{St}, (2)}$ such that for every event $A$,
$$ \tilde{\p}_{t}^{\text{St}, (2)}(A) := \frac{\E \textbf{1}_{A} \prod_{p \leq x^{1/e}} \left|1 - \frac{\alpha_ph(p)}{p^{1/2+\sigma}}\right|^{-2} \left|1 - \frac{\beta_ph(p)}{p^{1/2+\sigma}}\right|^{-2}\left|1 - \frac{\alpha_ph(p)}{p^{1/2+\sigma + it}}\right|^{-2}\left|1 - \frac{\beta_ph(p)}{p^{1/2+\sigma + it}}\right|^{-2}}{\E \prod_{p \leq x^{1/e}} \left|1 - \frac{\alpha_ph(p)}{p^{1/2+\sigma}}\right|^{-2} \left|1 - \frac{\beta_ph(p)}{p^{1/2+\sigma}}\right|^{-2}\left|1 - \frac{\alpha_ph(p)}{p^{1/2+\sigma + it}}\right|^{-2}\left|1 - \frac{\beta_ph(p)}{p^{1/2+\sigma + it}}\right|^{-2}}. $$
We also denote by $\tilde{\E}_{t}^{\text{St}, (2)}$ the expectation with respect to this measure. Here again the exact range of $p$ in the definition of the measure $\tilde{\p}_{t}^{\text{St}, (2)}$ does not matter too much. \newline

Using \eqref{Eest2} and a two dimensional version of the Berry--Esseen inequality, we obtain the following two dimensional analogue of in Lemma \ref{lem3}.
\begin{lemma}
\label{lem7}
With the notation as above. Suppose that $|t| \leq 1$. Suppose that $x_1 \geq e^{C/|t|^2}$ is sufficiently large, and that $|\sigma| \leq 1/\log x_n$. Suppose further that $(u_j)_{j=1}^{n}$ and $(v_j)_{j=1}^{n}$ are any sequences of real numbers satisfying
$$ |u_j| , |v_j| \leq (1/40)(\log x_j)^{1/4} + 2 \;\;\; \forall 1 \leq j \leq n . $$

Then we have
\begin{align}
\label{doubleprob}
\begin{split}
\tilde{\p}_{t}^{\text{St}, (2)} & (u_j \leq \log|I_{l_j}(\tfrac{1}{2} + \sigma)| \leq u_j + \tfrac{1}{j^2} , \;\;\; v_j \leq \log|I_{l_j}(\tfrac{1}{2} + \sigma + it)| \leq v_j + \tfrac{1}{j^2} ,  \; \forall \; j \leq n )  \\
 = & \left(1+O\left(\frac{1}{x_{1}^{1/100}}\right) \right) \p(u_j \leq N_j^{1} \leq u_j + \tfrac{1}{j^2} , \; \text{and} \; v_j \leq N_{j}^{2} \leq v_j + \tfrac{1}{j^2} , \; \forall \; 1 \leq j \leq n) ,
\end{split}
\end{align}
where $(N_j^{1}, N_j^{2})_{j=1}^{n}$ is a sequence of independent bivariate Gaussian random vectors, and the components $N_j^{1}, N_j^{2}$ have mean $\sum_{x_{j}^{1/e} < p \leq x_{j}} \frac{\lambda^2(p)(1 + \cos(t\log p))}{p^{1 + 2\sigma}}$, variance $\sum_{x_{j}^{1/e} < p \leq x_j} \frac{\lambda^2(p)}{2p^{1 + 2\sigma}}$, and covariance $\sum_{x_{j}^{1/e} < p \leq x_{j}} \frac{\lambda^2(p)\cos(t\log p)}{2p^{1 + 2\sigma}}$. \newline

Alternatively, we have
\begin{align*}
\tilde{\p}_{t}^{\text{St}, (2)} & (u_j \leq \log|I_{l_j}(\tfrac{1}{2} + \sigma)| \leq u_j + \tfrac{1}{j^2} , \;\;\; v_j \leq \log|I_{l_j}(\tfrac{1}{2} + \sigma + it)| \leq v_j + \tfrac{1}{j^2} ,  \; \forall \; j \leq n)  \\
& =  \left(1+O\left(\frac{1}{\sqrt{C}}\right) \right) \p(u_j \leq N_j^{1} \leq u_j + \tfrac{1}{j^2} \; \forall \; j \leq n) \cdot \p(v_j \leq N_{j}^{2} \leq v_j + \tfrac{1}{j^2}, \; \forall \; j \leq n).
\end{align*}
This means that we may replace the covariance of $N_j^{1}$, $N_j^{2}$ by zero.
\end{lemma}
\begin{proof}
 Our proof is a modification of that for \cite[Lemma 7]{Harper20}. As observed there, the probabilities on both sides of \eqref{doubleprob} factor as products over $j$ by independence.  Hence it suffices to show that for all $1 \leq j \leq n$,
\begin{align*}
\tilde{\p}_{t}^{\text{St}, (2)} & (u_j \leq \log|I_{l_j}(\tfrac{1}{2} + \sigma)| \leq u_j + \tfrac{1}{j^2} , \; \text{and} \; v_j \leq \log|I_{l_j}(\tfrac{1}{2} + \sigma + it)| \leq v_j + \tfrac{1}{j^2}) \\
& = \left(1+O\left(\frac{1}{x_{j}^{1/100}}\right) \right) \p(u_j \leq N_j^{1} \leq u_j + \tfrac{1}{j^2} , \; \text{and} \; v_j \leq N_{j}^{2} \leq v_j + \tfrac{1}{j^2}) \\
& = \left(1+O\left(\frac{1}{\sqrt{C e^j}}\right) \right) \p(u_j \leq N_j^{1} \leq u_j + \tfrac{1}{j^2}) \cdot \p(v_j \leq N_{j}^{2} \leq v_j + \tfrac{1}{j^2}).
\end{align*}
 By \eqref{Eest2}, the characteristic function $\tilde{\E}_{t}^{\text{St}, (2)} e^{iu\log|I_{l_j}(1/2 + \sigma)| + iv\log|I_{l_j}(1/2 + \sigma + it)|}$ is
\begin{align*}
%%\label{doubleprob}
\begin{split}
 & \exp\Big ( \sum_{x_{j}^{1/e} < p \leq x_j} \frac{\lambda^2(p)(iu - \frac{u^{2}}{4} + iv - \frac{v^{2}}{4})}{p^{1 + 2\sigma}} + \sum_{x_{j}^{1/e} < p \leq x_{j}} \frac{\lambda^2(p)(2iu + 2iv - uv)\cos(t\log p)}{2p^{1+2\sigma}} + T(u,v) - T(0,0) \Big ) \\
 =&  \exp\Big ( iuN_{j}^{1}+ivN_{j}^{2} + T(u,v) - T(0,0) \Big ),
\end{split}
\end{align*}
  where standard calculation shows that $N_{j}^{1}$, $N_{j}^{2}$ can be taken to be Gaussian random variables each having mean $\sum_{x_{j}^{1/e} < p \leq x_{j}} \frac{\lambda^2(p)(1 + \cos(t\log p))}{p^{1 + 2\sigma}}$, variance $\sum_{x_{j}^{1/e} < p \leq x_{j}} \frac{\lambda^2(p)}{2p^{1 + 2\sigma}}$, and covariance $\sum_{x_{j}^{1/e} < p \leq x_{j}} \frac{\lambda^2(p)\cos(t\log p)}{2p^{1 + 2\sigma}}$. \newline

  We deduce from the proof of \cite[Lemma 2.13]{G&Wu25-12} by setting $t_m=0$ there that
\begin{align}
\label{sumlambdasquareoverp}
\begin{split}
 \sum_{x_{j}^{1/e} < p \leq x_j} \frac{\lambda^2(p)}{p^{1+2\sigma}} = 1 + O\Big( \frac{1}{\log x_j} + \frac{\log x_j}{\log x_n} \Big).
\end{split}
\end{align}
  The above implies that we have in our case,
\begin{align}
\label{sumlambdasquareoverp1}
\begin{split}
 \sum_{x_{j}^{1/e} < p \leq x_j} \frac{\lambda^2(p)}{p^{1+2\sigma}} = 1 + O(C^{-1} e^{-j} + e^{-(n-j)}).
\end{split}
\end{align}

By \eqref{merten2-0} and \eqref{merten2}, we have for some constant $b_4>0$ and $x \geq 2$,
\begin{align}
\label{merten3}
\begin{split}
\sum_{p\le x} \lambda^2(p)\log p =&  x+O\Big(x \exp(-b_4\sqrt{\log x})\Big).
\end{split}
\end{align}

The above, by following the arguments in the proof of \cite[Theorem 6.9]{MVa1}, leads to,
\begin{align}
\label{merten4}
\begin{split}
 \pi_f(x):=\sum_{p\le x} \lambda^2(p)=&  \int\limits^x_2\frac {\dif u}{\log u}+O\Big(x \exp(-b_4\sqrt{\log x})\Big), \quad \mbox{for} \quad x \geq 2.
\end{split}
\end{align}
Now \eqref{merten4} and partial summation similar to those carried out in \cite[Section 6.1]{harpergp} yield that under our conditions $|t| \leq 1$, $x_j \geq x_1 \geq e^{C/|t|^2}$ and $|\sigma| \leq 1/\log x_n$, 
\begin{align}
\label{sumlambdacos}
\begin{split}
 \sum_{x_{j}^{1/e} < p \leq x_{j}} \frac{\lambda^2(p)\cos(t\log p)}{p^{1 + 2\sigma}} \ll \frac{1}{|t| \log x_j} \ll \frac{1}{e^j |t| \log x_1} \ll \frac{1}{C e^j}.
\end{split}
\end{align}

 Moreover, we deduce from \eqref{Eest2} that we have
\begin{align}
\label{Tdiffest}
\begin{split}
 |T(u,v) - T(0,0)| \ll & \frac{1 + |u|^3 + |v|^3}{\sqrt{x_{j}^{1/e}} \log x_j},  \; |T(u,0) - T(0,0)| \ll  \frac{|u| + |u|^3}{\sqrt{x_{j}^{1/e}} \log x_j}, \; |T(0,v) - T(0,0)| \ll \frac{|v| + |v|^3}{\sqrt{x_{j}^{1/e}} \log x_j}, \\
 |T(u,v) - T(u,0) & - T(0,v) + T(0,0)| = \left| \int\limits_{0}^{u} \int\limits_{0}^{v} \frac{\partial T(u,v)}{\partial u \partial v} \dif u \dif v \right| \ll \frac{|u| |v| (1 + |u| + |v|)}{\sqrt{x_{j}^{1/e}} \log x_j} .
\end{split}
\end{align}

Now by the two dimensional Berry--Esseen inequality (see \cite{Sadikova}), we have
\begin{align*}
& \Biggl| \tilde{\p}_{t}^{\text{St}, (2)}(u_j \leq \log|I_{l_j}(\tfrac{1}{2} + \sigma)| \leq u_j + \tfrac{1}{j^2} , \; \text{and} \; v_j \leq \log|I_{l_j}(\tfrac{1}{2} + \sigma + it)| \leq v_j + \tfrac{1}{j^2})  \\
& \hspace*{2cm} - \p(u_j \leq N_{j}^{1} \leq u_j + \tfrac{1}{j^2} , \; \text{and} \; v_j \leq N_{j}^{2} \leq v_j + \tfrac{1}{j^2}) \Biggr| \\
& \ll  \int\limits_{-x_{j}^{1/50}}^{x_{j}^{1/50}} \int\limits_{-x_{j}^{1/50}}^{x_{j}^{1/50}} \Biggl|\frac{\Delta(u,v)}{uv}\Biggr| \dif u \dif v + \int\limits_{-x_{j}^{1/50}}^{x_{j}^{1/50}} \Biggl|\frac{\tilde{\E}_{t}^{\text{St}, (2)} e^{iu\log|I_{l_j}(1/2 + \sigma)|} - \E e^{iu N_{j}^{1}} }{u}\Biggr| \dif u  \\
& \hspace*{2cm} + \int\limits_{-x_{j}^{1/50}}^{x_{j}^{1/50}} \Biggl|\frac{\tilde{\E}_{t}^{\text{St}, (2)} e^{iv\log|I_{l_j}(1/2 + \sigma + it)|} - \E e^{iv N_{j}^{2}} }{v}\Biggr| \dif v + \frac{1}{x_{j}^{1/50}} ,
\end{align*}
where
\begin{align*}
\Delta(u,v)  := & \tilde{\E}_{t}^{\text{St}, (2)} e^{iu\log|I_{l_j}(1/2 + \sigma)| + iv\log|I_{l_j}(1/2 + \sigma + it)|} - \E e^{iuN_{j}^{1} + ivN_{j}^{2}} \nonumber \\
& \hspace*{2cm} - \tilde{\E}_{t}^{\text{St}, (2)} e^{iu\log|I_{l_j}(1/2 + \sigma)|} \tilde{\E}_{t}^{\text{St}, (2)} e^{iv\log|I_{l_j}(1/2 + \sigma + it)|} + \E e^{iuN_{j}^{1}} \E e^{ivN_{j}^{2}} . \nonumber
\end{align*}
 We then proceed as in the proof of \cite[Lemma 7]{Harper20} upon making use of the estimates in \eqref{sumlambdasquareoverp1}, \eqref{sumlambdacos} and \eqref{Tdiffest} to see that the assertions of the lemma are valid. This completes the proof.
\end{proof}

 The above lemma now allows us to obtain the following analogue of \cite[Proposition 7]{Harper20}.
\begin{prop}
\label{prop7}
With the notation as above, there is a large natural number $B$ so that the following is valid.  Let $t \in \mr$ satisfy $|t| \leq 1$ and $D \geq 2\log(1/|t|) + (B+1)$ be any natural number.  Suppose $n \leq \log\log x - D$ be large, and define the decreasing sequence $(l_j)_{j=1}^{n}$ of non-negative integers by $l_j := \lfloor \log\log x \rfloor - D - j$. Suppose also that $|\sigma| \leq 1/e^{D+n}$.  Then uniformly for any large $a$ and any function $H(n)$ satisfying $|H(n)| \leq 10\log n$, and with $I_{l}(s)$ being defined as in \eqref{Ildef} corresponding to a Steinhaus random multiplicative function, we have
\[ \tilde{\p}_{t}^{\text{St}, (2)}(-a -Bj \leq \sum_{m=1}^{j} \log|I_{l_m}(\tfrac{1}{2} + \sigma)| , \sum_{m=1}^{j} \log|I_{l_m}(\tfrac{1}{2} + \sigma + it)| \leq a + j + H(j), \; \forall \; j \leq n) \asymp \min\left\{1,\frac{a}{\sqrt{n}}\right\}^2 . \]
\end{prop}
\begin{proof}
 The proof proceeds along the same line as done in the proof of \cite[Proposition 7]{Harper20}, upon replacing the $(G_m)_{m=1}^{j}$ there by a sequence of independent Gaussian random variables with mean zero, variance $\sum_{x_{m}^{1/e} < p \leq x_m} \frac{\lambda^2(p)}{2p^{1 + 2\sigma}}$ and using Lemma \ref{lem7}. This completes the proof.
\end{proof}

  We now give a Rademacher counterpart of Lemma \ref{lem4}, which is an analogue of \cite[Lemma 5]{Harper20}.
\begin{lemma}
\label{lem5}
With the notation as above and $t \in \mr\backslash{\{0\}}$, suppose that $x_1 \geq \max\{e^{C/|t|},e^{C\log^{2}|t|}\}$ is large, $|\sigma| \leq 1/\log x_n$, and that the numbers $(t_j)_{j=1}^{n}$ satisfy $|t_j - t| \leq \frac{1}{j^{2/3} \log x_j}$.  Moreover, let $(u_j)_{j=1}^{n}$ and $(v_j)_{j=1}^{n}$ be sequences of real numbers such that for any $1 \leq j \leq n$,
$$ -(1/80)\sqrt{\log x_j} \leq u_j \leq v_j \leq (1/80)\sqrt{\log x_j}. $$
Then we have
\begin{align*}
 \p((u_j - j) + O(1) \leq \sum_{m=1}^{j} G_m \leq (v_j-j) - O(1) \; \forall j \leq n) & \ll  \tilde{\p}_{t}^{\text{Rad}}(u_j \leq \sum_{m=1}^{j} \log|I_{l_m}(1/2 + \sigma + it_m)| \leq v_j \; \forall j \leq n)  \\
& \ll  \p((u_j - j) - O(1) \leq \sum_{m=1}^{j} G_m \leq (v_j-j) + O(1) \; \forall j \leq n) ,
\end{align*}
where $G_m$ are independent Gaussian random variables, each having mean 0 and variance $\sum_{x_{m}^{1/e} < p \leq x_m} \frac{\lambda^2(p)(1+\cos(2t_m \log p))}{2p^{1 + 2\sigma}}$.
\end{lemma}
\begin{proof}
Our proof of similar to that of \cite[Lemma 5]{Harper20}. Lemma \ref{eulerproductR} implies that for any $|u| \leq x_{j}^{1/20}$,
$$ \tilde{\E}_{t}^{\text{Rad}} e^{iu\log|I_{l_j}(1/2 + \sigma + it_j)|}= \exp\Big (\sum_{x_{j}^{1/e} < p \leq x_j} \frac{\lambda^2(p)(iuc(t , t_j - t , p) - (u^{2}/4)(1+\cos(2t_j \log p)))}{p^{1 + 2\sigma}} + \tilde T(u) - \tilde T(0)\Big ) , $$
where
\begin{align}
\label{cdef}
\begin{split}
c(t , t_j - t , p) = 2\cos(t \log p) \cos(t_j\log p) - \tfrac12 \cos(2 t_j \log p).
\end{split}
\end{align}
  As shown in the proof of \cite[Lemma 5]{Harper20} that our assumption on $|t_j - t|$ implies that
\begin{align*}
%%\label{cosidentity}
\begin{split}
2\cos(t \log p) \cos(t_j\log p) & = \cos((t+t_j)\log p) + 1 + O\Big( \frac {1}{j^{4/3}} \Big).
\end{split}
\end{align*}

  As $x_{j}^{1/e} \geq x_{1}^{1/e} \geq \max\{e^{C/e|t|}, e^{(C/e)\log^{2}|t|}\}$ is large, we obtain similar to \eqref{sumlambdacos} upon using \eqref{merten4} and partial summation similar to see that
\begin{align}
\label{sumlambdasquarecosoverp}
\begin{split}
\sum_{x_{j}^{1/e} < p \leq x_j} \frac{\lambda^2(p)\cos((t+t_j)\log p)}{p^{1 + 2\sigma}} , \;\;\; \sum_{x_{j}^{1/e} < p \leq x_j} \frac{\lambda^2(p)\cos(2 t_j \log p)}{p^{1 + 2\sigma}} \ll \frac{1}{|t|\log x_j} \ll \frac{1}{e^{j}|t| \log x_1} \ll \frac{1}{Ce^{j}} .
\end{split}
\end{align}

We now proceed as in the proofs of \cite[Lemma 3--Lemma 4]{Harper20}.  The relevant Gaussian random variables $N_j$ in our case have mean $\sum_{x_{j}^{1/e} < p \leq x_j} \frac{\lambda^2(p)c(t , t_j - t , p)}{p^{1 + 2\sigma}}$ and variance $\sum_{x_{j}^{1/e} < p \leq x_j} \frac{\lambda^2(p)(1+\cos(2t_j \log p))}{2p^{1 + 2\sigma}}$. Moreover, by \eqref{lambdabound}, \eqref{sumlambdasquareoverp}, \eqref{cdef}, \eqref{sumlambdasquarecosoverp} and our assumptions on $t_j , x_j$, we deduce that
\begin{align*}
%%\label{sumlambdasquarecoverp}
\begin{split}
\sum_{x_{j}^{1/e} < p \leq x_j} \frac{\lambda^2(p)c(t , t_j - t , p)}{p^{1 + 2\sigma}} =& 1 + O\Big(\frac{1}{\log x_j} + \frac{\log x_j}{\log x_n} + \frac{1}{j^{4/3}}\Big) = 1 + O\Big(\frac{1}{e^{n-j}} + \frac{1}{j^{4/3}}\Big), \\
\sum_{x_{j}^{1/e} < p \leq x_j} \frac{\lambda^2(p)(1+\cos(2t_j \log p))}{2p^{1 + 2\sigma}} =& \sum_{x_{j}^{1/e} < p \leq x_j} \frac{\lambda^2(p)}{2p^{1 + 2\sigma}} + O\Big(\frac 1C\Big).
\end{split}
\end{align*}
 We then follow the treatments in the proofs of \cite[Lemma 3--Lemma 4]{Harper20} using the above formulas to get that the assertion of the lemma is valid. This completes the proof.
\end{proof}

 We apply Lemma \ref{lem5} with Probability Results 1 and 2 in \cite{Harper20} and argue as in the proof of \cite[Proposition 6]{Harper20} to deduce the following Rademacher analogue of Proposition \ref{prop5}.
\begin{prop}
\label{prop6}
With the notation as above, there exists a large natural number $B$ so that the following is valid.  Let $t \in \mr\backslash\{0\}$, $D \geq \max\{\log(1/|t|), 2\log\log(1+|t|)\} + (B+1)$ be any natural number, $n \leq \log\log x - D$ large.  Define the decreasing sequence $(l_j)_{j=1}^{n}$ of non-negative integers by $l_j := \lfloor \log\log x \rfloor - D - j$. Suppose that $|\sigma| \leq 1/e^{D+n}$, and that $(t_j)_{j=1}^{n}$ is a real sequence satisfying $|t_j - t| \leq 1/(j^{2/3} e^{D+j})$ for all $j$.  Then uniformly for any large $a$ and any function $H(n)$ satisfying $|H(n)| \leq 10\log n$, and with $I_{l}(s)$ defined in \eqref{Ildef} corresponding to a Rademacher random multiplicative function, we have
$$ \tilde{\p}_{t}^{\text{Rad}}(-a -Bj \leq \sum_{m=1}^{j} \log|I_{l_m}(\tfrac12 + \sigma + it_m)| \leq a + j + H(j) \; \forall j \leq n) \asymp \min\left\{1,\frac{a}{\sqrt{n}}\right\} . $$
\end{prop}

%%--------------------------------------------------------------------------------------------------
%%--------------------------------------------------------------------------------------------------

\section{Bounding $\E|\sum_{n \leq x} h(n)\lambda(n)|^{2q}$ by Euler products}

We now start our proof of Theorem \ref{lowerboundsfixedmodmean}.  Using the key ideas in \cite{Harper20}, we bound $\E|\sum_{n \leq x} h(n)\lambda(n)|^{2q}$ by certain Euler products.

\subsection{The upper bounds}
\label{secupperbounds}

For a Steinhaus or Rademacher random multiplicative function $h(n)$, any real number $P \geq 2$ and $s \in \comc$ with $\Re(s) > 0$, let $F_{P, f}$ denote the partial Euler product of $h(n)\lambda(n)$ over $P$-smooth numbers.  Thus, we have by \eqref{Lambdapkrel},
\begin{align}
\label{FPdef}
\begin{split}
F_{P,f}(s) =& \sum_{\substack{n=1, \\ n \; \text{is} \; P \; \text{smooth}}}^{\infty} \frac{h(n)\lambda(n)}{n^s} \\
=& \begin{cases}
\displaystyle \prod_{p \leq P} \left(1 - \frac{\alpha_ph(p)}{p^s}\right)^{-1}\left(1 - \frac{\beta_ph(p)}{p^s}\right)^{-1},  \quad & \text{$h$ is a Steinhaus random multiplicative function}, \\
\displaystyle \prod_{p \leq P} \left( 1 + \frac{\lambda(p)h(p)}{p^s} \right), \quad & \text{$h$ is a Rademacher random multiplicative function}.
\end{cases}
\end{split}
\end{align}
  We also write, for brevity, $F_{k,f}$ for the function $F_{x^{e^{-(k+1)}}, f}$ for any integer $-1 \leq k \leq \log\log x $.
We also define $\| \cdot \|_{r} := (\E|\cdot|^{r})^{1/r}$ for $r \geq 1$. \newline

The following result estimating $\E|\sum_{n \leq x} h(n)\lambda(n)|^{2q}$ from above is analogous to \cite[Propositions 1--2]{Harper20}.
\begin{prop}
\label{propupperboundsSteinhaus}
With the notation as above, let $x$ be large and set $\mathcal{K} := \lfloor \log\log\log x \rfloor$. Then we have, uniformly for all $2/3 \leq q \leq 1$ and any Steinhaus random multiplicative function $h(n)$,
\begin{align}
\label{EupperboundS}
\begin{split}
 \Big\| \sum_{n \leq x} h(n)\lambda(n) \Big\|_{2q} \ll \sqrt{\frac{x}{\log x}} \sum_{0 \leq k \leq \mathcal{K}} \Big\| \ \int\limits_{-1/2}^{1/2} |F_{k,f}(\tfrac12 - \tfrac{k}{\log x} + it)|^2 \ \dif t \Big\|_{q}^{1/2} + \sqrt{\frac{x}{\log \log x}} .
\end{split}
\end{align}
 Similarly, we have uniformly for all $2/3 \leq q \leq 1$ and any Rademacher random multiplicative function $h(n)$,
\begin{align}
\label{EupperboundR}
\begin{split}
\Big\| \sum_{n \leq x} h(n)\lambda(n) \Big\|_{2q}  \ll & \sqrt{\frac{x}{\log x}} \sum_{0 \leq k \leq \mathcal{K}} \max_{N \in \mz} \frac{1}{(|N|+1)^{1/8}} \Big\| \int\limits_{N-1/2}^{N+1/2} |F_{k,f}(\tfrac12 - \tfrac{k}{\log x} + it)|^2 \ \dif t \Big\|_{q}^{1/2} + \sqrt{\frac{x}{\log \log x}} .
\end{split}
\end{align}
\end{prop}
\begin{proof}
 Denote by $P(n)$ the largest prime factor of $n$. For $h(n)$ either a Steinhaus or a Rademacher random multiplicative function, Minkowski's inequality yields, for all $2/3 \leq q \leq 1$,
\begin{align}
\label{EPndecomp}
 \Big\| \sum_{n \leq x} h(n)\lambda(n) \Big\|_{2q} \leq \sum_{0 \leq k \leq \mathcal{K}} \Big\| \sum_{\substack{n \leq x, \\ x^{e^{-(k+1)}} < P(n) \leq x^{e^{-k}}}} h(n)\lambda(n) \Big\|_{2q} + \Big\| \sum_{\substack{n \leq x, \\ P(n) \leq x^{e^{-(\mathcal{K}+1)}} }} h(n)\lambda(n) \Big\|_{2q} . 
 \end{align}
Then H\"{o}lder's inequality and the orthogonality of random multiplicative functions (recall $\mathcal{K} = \lfloor \log\log\log x \rfloor$) lead to
\begin{align}
\label{EPnsmall}
 \Big\| \sum_{\substack{n \leq x, \\ P(n) \leq x^{e^{-(\mathcal{K}+1)}} }} h(n)\lambda(n) \Big\|_{2q} \leq \Big\| \sum_{\substack{n \leq x, \\ P(n) \leq x^{e^{-(\mathcal{K}+1)}} }} h(n)\lambda(n) \Big\|_{2}  \leq \Biggl(\sum_{n \leq x, P(n) \leq x^{1/\log\log x}}\lambda(n)^{2}\Biggr)^{1/2} \ll \sqrt{\frac{x}{\log \log x}},
\end{align}
 where the last bound above is given in \cite[Section 3.1]{G&Wu25-12}. \newline

Let $\E^{(k)}$ denote expectation conditional on $(h(p))_{p \leq x^{e^{-(k+1)}}}$.  H\"{o}lder's inequality and a mean square calculation render 
\begin{align*}
\sum_{0 \leq k \leq \mathcal{K}} \Big\|  \sum_{\substack{n \leq x, \\ x^{e^{-(k+1)}} < P(n) \leq x^{e^{-k}}}} & h(n)\lambda(n) \Big\|_{2q}  =  \sum_{0 \leq k \leq \mathcal{K}} \Big\| \sum_{\substack{1 < m \leq x , \\ p|m \Rightarrow x^{e^{-(k+1)}} < p \leq x^{e^{-k}}}} h(m)\lambda(m) \sum_{\substack{n \leq x/m, \\ n \; \text{is} \; x^{e^{-(k+1)}} \text{-smooth}}} h(n)\lambda(n)   \Big\|_{2q}  \\
& = \sum_{0 \leq k \leq \mathcal{K}} \Bigl( \E \E^{(k)}\Bigl|\sum_{\substack{1 < m \leq x , \\ p|m \Rightarrow x^{e^{-(k+1)}} < p \leq x^{e^{-k}}}} h(m)\lambda(m) \sum_{\substack{n \leq x/m, \\ n \; \text{is} \; x^{e^{-(k+1)}} \text{-smooth}}} h(n)\lambda(n) \Bigr|^{2q} \Bigr)^{1/2q} \\
& \leq \sum_{0 \leq k \leq \mathcal{K}} \Bigl( \E \Big( \E^{(k)}\Bigl|\sum_{\substack{1 < m \leq x , \\ p|m \Rightarrow x^{e^{-(k+1)}} < p \leq x^{e^{-k}}}} h(m)\lambda(m)  \sum_{\substack{n \leq x/m, \\ n \; \text{is} \; x^{e^{-(k+1)}} \text{-smooth}}} h(n)\lambda(n)  \Bigr|^2 \Big)^{q} \Bigr)^{1/2q} \\
& = \sum_{0 \leq k \leq \mathcal{K}} \Big\| \sum_{\substack{1 < m \leq x , \\ p|m \Rightarrow x^{e^{-(k+1)}} < p \leq x^{e^{-k}}}} \lambda^2(m) \Bigl|\sum_{\substack{n \leq x/m, \\ n \; \text{is} \; x^{e^{-(k+1)}} \text{-smooth}}} h(n)\lambda(n)  \Bigr|^2 \Big\|_{q}^{1/2} .
\end{align*}

Now set $X = e^{b_4\sqrt{\log x}/4}$ where $b_4$ is given in \eqref{merten3}.  We have uniformly, for any $2/3 \leq q \leq 1$ and $0 \leq k \leq \mathcal{K}$,
\begin{align*}
%%\label{Erandnlarge000}
\begin{split}
\E \Bigl(\sum_{\substack{1 < m \leq x , \\ p|m \Rightarrow x^{e^{-(k+1)}} < p \leq x^{e^{-k}}}} & \lambda^2(m) \Bigl| \sum_{\substack{n \leq x/m, \\ n \; \text{is} \; x^{e^{-(k+1)}} \text{-smooth}}} h(n)\lambda(n)  \Bigr|^2 \Bigr)^{q}  \\
 \ll & \E \Bigl( \sum_{\substack{1 < m \leq x , \\ p|m \Rightarrow x^{e^{-(k+1)}} < p \leq x^{e^{-k}}}} \frac{\lambda^2(m)X}{m} \int\limits_{m}^{m(1 + \frac{1}{X})} \Bigl| \sum_{\substack{n \leq x/t, \\ x^{e^{-(k+1)}} \text{-smooth}}} h(n)\lambda(n) \Bigr|^2 \dif t \Bigr)^{q} \\
& \hspace*{2cm} + \E \Bigl(\sum_{\substack{1 < m \leq x , \\ p|m \Rightarrow x^{e^{-(k+1)}} < p \leq x^{e^{-k}}}} \frac{\lambda^2(m)X}{m} \int\limits_{m}^{m(1 + \frac{1}{X})} \Bigl| \sum_{\substack{x/t < n \leq x/m, \\ x^{e^{-(k+1)}} \text{-smooth}}} h(n)\lambda(n) \Bigr|^2 \dif t \Bigr)^{q} \\
 =: & E_1+E_2, \quad \mbox{say}.
\end{split}
\end{align*}

From \eqref{EPndecomp}, \eqref{EPnsmall} and the above that
\begin{align}
\label{EPndecomp1}
 \Big\| \sum_{n \leq x} h(n)\lambda(n) \Big\|_{2q} \ll \sum_{0 \leq k \leq \mathcal{K}} (E_1^{1/(2q)}+E_2^{1/(2q)}) + \sqrt{\frac{x}{\log \log x}}.
\end{align}

H\"{o}lder's inequality and a mean square calculation yield
\begin{align}
\label{Erandnlarge0}
\begin{split}
E_2 \ll & \Bigl(\sum_{\substack{1 < m \leq x , \\ p|m \Rightarrow x^{e^{-(k+1)}} < p \leq x^{e^{-k}}}} \frac{\lambda^2(m)X}{m} \int\limits_{m}^{m(1 + \frac{1}{X})} \E \Bigl| \sum_{\substack{x/t < n \leq x/m, \\ x^{e^{-(k+1)}} \text{-smooth}}} h(n)\lambda(n) \Bigr|^2 \dif t \Bigr)^{q} \\
% =& \Biggl(\sum_{\substack{1 < m \leq x , \\ p|m \Rightarrow x^{e^{-(k+1)}} < p \leq x^{e^{-k}}}} \frac{\lambda^2(m)X}{m} \int\limits_{m}^{m(1+1/X)} \sum_{\substack{x/t < n \leq x/m, \\ x^{e^{-(k+1)}} \text{-smooth}}} \lambda(n)^{2} \dif t \Biggr)^q  \\
\ll & \Bigl(\sum_{\substack{1 < m \leq x , \\ p|m \Rightarrow x^{e^{-(k+1)}} < p \leq x^{e^{-k}}}}  \frac{\lambda^2(m)X}{m} \int\limits_{m}^{m(1+1/X)} \sum_{\substack{x/t < n \leq x/m}} \lambda(n)^{2}  \dif t \Bigr)^q  \\
\ll & \Bigl(\sum_{\substack{1 < m \leq x , \\ p|m \Rightarrow x^{e^{-(k+1)}} < p \leq x^{e^{-k}}}}  \frac{\lambda^2(m)X}{m} \int\limits_{m}^{m(1+1/X)} \frac xm-\frac xt+\Big( \frac x{m} \Big)^{3/5}  \dif t \Bigr)^q,
\end{split}
\end{align}
  where the last bound above follows from \eqref{lambdasquareasymp}.  If $m \leq t \leq m(1+1/X)$, then $\frac xm-\frac xt \leq \frac xm-\frac x{m(1+1/X)} \ll x/(mX)$. It follows that the last expression in \eqref{Erandnlarge0} is
\begin{align}
\label{EEjremaindersimplified}
\begin{split}
 %\ll & \Biggl(\sum_{\substack{1 < m \leq x , \\ p|m \Rightarrow x^{e^{-(k+1)}} < p \leq x^{e^{-k}}}}  \frac{\lambda^2(m)X}{m} \int\limits_{m}^{m(1+1/X)} \frac x{mX}+\Big( \frac x{m} \Big)^{3/5}  \dif t \Biggr)^q
 \ll  \Bigl(\sum_{\substack{1 < m \leq x , \\ p|m \Rightarrow x^{e^{-(k+1)}} < p \leq x^{e^{-k}}}}  \lambda^2(m)(\frac{x}{mX} + \Big( \frac x{m} \Big)^{3/5} \Big) \Bigr)^q.
\end{split}
\end{align}

  By \eqref{lambdasquareasymp} and partial summation,
\begin{align*}
%%\label{EEjremainder}
\begin{split}
 \sum_{\substack{m \leq x}} \frac{\lambda^2(m)}{m}\ll \log x.
\end{split}
\end{align*}
It follows from this and $X= e^{b_4\sqrt{\log x}/4}$ that for, $0 \leq k \leq \mathcal{K}$,
\begin{align}
\label{EEjremainder10}
\begin{split}
 \sum_{\substack{1 < m \leq x , \\ p|m \Rightarrow x^{e^{-(k+1)}} < p \leq x^{e^{-k}}}}  \frac{x\lambda^2(m)}{mX} \ll \frac{x}{X} \sum_{\substack{m \leq x}} \frac{\lambda^2(m)}{m} \ll \frac {x\log x}{X} \ll 2^{-e^{k}} \frac{x}{\log x}.
\end{split}
\end{align}

Next, for $x^{1/10} < v \leq x$, we have
$$ \sum_{\substack{m \leq v, \\ p | m \Rightarrow x^{e^{-(k+1)}} \leq p \leq x^{e^{-k}}}} \lambda^2(m) \leq x^{1/20} + (5^{20})^{-e^{k}/20}  \sum_{\substack{x^{1/20} \leq m \leq v, \\ p | m \Rightarrow x^{e^{-(k+1)}} \leq p \leq x^{e^{-k}}}} \lambda^2(m)(5^{20})^{\Omega(m)} , $$
since a number $d \geq x^{1/20}$ with all its prime factors smaller than $x^{e^{-k}}$ must have at least $e^{k}/20$ such factors, counted with multiplicity.  Moreover, by \eqref{lambdabound} we have $\lambda(m)^2 \leq d^2(m)$.  Writing $m=\prod p^{k_i}_i$ with $p_i$ being distinct primes and $k_i \geq 1$, then by the well-known inequality $1+x \leq e^x$ for all real $x$,  we have $d(m)=\prod(1+k_i) \leq \prod e^{k_i}=e^{\Omega(m)}$. It follows that
$$ \sum_{\substack{x^{1/20} \leq m \leq v, \\ p | m \Rightarrow x^{e^{-(k+1)}} \leq p \leq x^{e^{-k}}}} \lambda^2(m)(5^{20})^{\Omega(m)} \leq \sum_{\substack{x^{1/20} \leq m \leq v, \\ p | m \Rightarrow x^{e^{-(k+1)}} \leq p \leq x^{e^{-k}}}} (5^{20}e)^{\Omega(m)} , $$

 We now apply Number Theory Result 1 in \cite{Harper20}, setting the parameters there $v=v$, $u=x^{1/20}$, $y=x^{e^{-(k+1)}}$, $z=x^{e^{-k}}$, $\delta = 1/2$.  This leads to
\begin{eqnarray}
(5^{20})^{-e^{k}/20}  \sum_{\substack{x^{1/20} \leq m \leq v, \\ p | m \Rightarrow x^{e^{-(k+1)}} \leq p \leq x^{e^{-k}}}} \lambda^2(m)(5^{20})^{\Omega(m)} & \ll & 5^{-e^{k}} \frac{v}{\log(x^{e^{-(k+1)}})} \prod_{x^{e^{-(k+1)}} \leq p \leq x^{e^{-k}}}\left(1 - \frac{5^{20}e}{p}\right)^{-1} \nonumber \\
& \ll & 5^{-e^{k}} e^{k} \frac{v}{\log x} \ll 2^{-e^{k}} \frac{v}{\log x} . \nonumber
\end{eqnarray}
As $2^{-e^{k}} v/\log x$ is always greater than $x^{1/20}$ on our range of $k$, we conclude that for $x^{1/20} \leq v \leq x$,
\begin{equation}
\label{lambdasquareplarge0}
    \sum_{\substack{m \leq v, \\ p | m \Rightarrow x^{e^{-(k+1)}} \leq p \leq x^{e^{-k}}}} \lambda^2(m)  \ll 2^{-e^{k}} \frac{v}{\log x}.
\end{equation}
The above and partial summation render
\begin{align}
\label{EEjremainder21-1}
\begin{split}
  \sum_{\substack{x^{1/10} < m \leq x, \\ p | m \Rightarrow x^{e^{-(k+1)}} \leq p \leq x^{e^{-k}}}} \lambda^2(m)\Big( \frac x{m} \Big)^{3/5} \ll 2^{-e^{k}} \frac{x}{\log x}.
\end{split}
\end{align}
  On the other hand, \eqref{lambdabound} and summing trivially give that, for $0 \leq k \leq \mathcal{K}$,
\begin{align}
\label{EEjremainder22-2}
\begin{split}
   \sum_{\substack{1  < m \leq x^{1/10}, \\ p | m \Rightarrow x^{e^{-(k+1)}} \leq p \leq x^{e^{-k}}}} \lambda^2(m)\Big( \frac x{m} \Big)^{3/5} \ll  x^{3/5}x^{1/10} \ll  2^{-e^{k}} \frac{x}{\log x}.
\end{split}
\end{align}

Thus, we deduce from \eqref{EEjremainder21-1} and \eqref{EEjremainder22-2} that for $0 \leq k \leq \mathcal{K}$,
\begin{align}
\label{EEjremainder23}
\begin{split}
   \sum_{\substack{1  < m \leq x, \\ p | m \Rightarrow x^{e^{-(k+1)}} \leq p \leq x^{e^{-k}}}} \lambda^2(m) \Big(\frac x{m}\Big)^{3/5}  \ll 2^{-e^{k}} \frac{x}{\log x}.
\end{split}
\end{align}

From \eqref{Erandnlarge0}--\eqref{EEjremainder10} and \eqref{EEjremainder23} that
\begin{align*}
%%\label{Enlarge}
\begin{split}
E_2 \ll  \Biggl(2^{-e^{k}} \frac{x}{\log x} \Biggr)^{q}.
\end{split}
\end{align*}

  Upon taking the $2q$-th root and summing over $0 \leq k \leq \mathcal{K}$, we see that
\begin{align*}
%%\label{sumE2}
 \sum_{0 \leq k \leq \mathcal{K}} E_2^{1/(2q)} \ll \sqrt{\frac{x}{\log x}}.
\end{align*}

   In view of the above and \eqref{EPndecomp1}, we see that in order to establish our proposition, it remains to show that
\begin{align}
\label{sumE1}
\begin{split}
 \sum_{0 \leq k \leq \mathcal{K}} E_1^{1/(2q)} \ll
\begin{cases}
\displaystyle \sqrt{\frac{x}{\log x}} \sum_{0 \leq k \leq \mathcal{K}} \Big\| \int\limits_{-1/2}^{1/2} |F_{k,f}( \tfrac12 - \tfrac{k}{\log x} + it)|^2 \dif t \Big\|_{q}^{1/2}, \quad &  \text{if $h$ is Steinhaus}, \\
\displaystyle \sqrt{\frac{x}{\log x}} \sum_{0 \leq k \leq \mathcal{K}} \max_{N \in \mz} \frac{1}{(|N|+1)^{1/8}} \Big\| \int\limits_{N-1/2}^{N+1/2} |F_{k,f}( \tfrac12 - \tfrac{k}{\log x} + it)|^2 \dif t  \Big\|_{q}^{1/2} , \quad & \text{if $h$ is Rademacher}.
\end{cases}
\end{split}
\end{align}

  To this end, we interchange the sum and integral to arrive at
\begin{align}
\label{Erandnlarge00}
E_1 \ll  \E \Biggl( \int_{x^{e^{-(k+1)}}}^{x} \Biggl| \sum_{\substack{n \leq x/t, \\ x^{e^{-(k+1)}} \text{-smooth}}} h(n)\lambda(n) \Biggr|^2 \sum_{\substack{t/(1+1/X) \leq m \leq t , \\ p|m \Rightarrow x^{e^{-(k+1)}} < p \leq x^{e^{-k}}}} \frac{\lambda^2(m)X}{m} \dif t \Biggr)^{q}.
\end{align}

    We now estimate the expression
$$\sum_{\substack{t/(1+1/X) \leq m \leq t , \\ p|m \Rightarrow x^{e^{-(k+1)}} < p \leq x^{e^{-k}}}} \lambda^2(m)$$
 by arguments similar to those in our discussions above that lead to the bound in \eqref{lambdasquareplarge0}. In this process, we need to apply Number Theory Result 1 in \cite{Harper20}, setting the parameters there $v=t$, $u=t/(1+1/X)$, $y=x^{e^{-(k+1)}}$, $z=x^{e^{-k}}$, $\delta = 1/2$.  In our case the condition $1<u<v(1-y^{-1/2})$ needed by Number Theory Result 1 in \cite{Harper20} is satisfied as $X=e^{b_4\sqrt{\log x}/4}$. It follows that for $t \leq x$,
\begin{align*}
  \sum_{\substack{t/(1+1/X) \leq m \leq t , \\ p|m \Rightarrow x^{e^{-(k+1)}} < p \leq x^{e^{-k}}}} \lambda^2(m) \ll \frac{2^{-e^{k}}}{\log x} \Big(t-\frac {t}{1+1/X} \Big) \ll \frac{t}{X\log t}.
\end{align*}

Consequently, for $t \leq x$,
\begin{align*}
  \sum_{\substack{t/(1+1/X) \leq m \leq t , \\ p|m \Rightarrow x^{e^{-(k+1)}} < p \leq x^{e^{-k}}}} \frac{\lambda^2(m)X}{m(1+1/X)} \ll \frac{X}{t(1+1/X)}\sum_{\substack{t/(1+1/X) \leq m \leq t , \\ p|m \Rightarrow x^{e^{-(k+1)}} < p \leq x^{e^{-k}}}} \lambda^2(m) \ll \frac 1{\log t}.
\end{align*}

From this and \eqref{Erandnlarge00},
\begin{align*}
%%\label{Erandnlarge0000}
\begin{split}
 E_1 
 %=& \E \Biggl( \sum_{\substack{1 < m \leq x , \\ p|m \Rightarrow x^{e^{-(k+1)}} < p \leq x^{e^{-k}}}} \frac{\lambda^2(m)X}{m} \int_{m}^{m(1 + \frac{1}{X})} \Biggl| \sum_{\substack{n \leq x/t, \\ x^{e^{-(k+1)}} \text{-smooth}}} h(n)\lambda(n) \Biggr|^2 dt \Biggr)^{q} \\
\ll & \E \Bigl( \int\limits_{x^{e^{-(k+1)}}}^{x} \Bigl| \sum_{\substack{n \leq x/t, \\ x^{e^{-(k+1)}} \text{-smooth}}} h(n)\lambda(n) \Bigr|^2 \frac{\dif t}{\log t} \Bigr)^q.
\end{split}
\end{align*}
  We proceed from here by following the arguments in the proof of \cite[Proposition 1]{Harper20}.  This lead to \eqref{sumE1} and completes the proof.
\end{proof}

\subsection{The lower bounds}
\label{lowerbounds}

 We write $F_f$ for the function $F_{-1,f}$ defined in Section \ref{secupperbounds}, so that $F_f$ denotes the partial Euler product of $h(n)\lambda(n)$ over $x$-smooth numbers, for $h$ being either a Steinhaus or a Rademacher random multiplicative function.  The lower bounds for $\E|\sum_{n \leq x} h(n)\lambda(n)|^{2q}$, analogous to \cite[Propositions 3--4]{Harper20}, is the following.
\begin{prop}
\label{prop3}
 With the notation as above, let $h(n)$ be a Steinhaus or a Rademacher random multiplicative function. There exists a large absolute constant $C > 0$ such that, uniformly for all $2/3 \leq q \leq 1$ and large real $x$,
\begin{align}
\label{hlambdalower}
\begin{split}
 \Big\| \sum_{n \leq x} h(n)\lambda(n) \Big\|_{2q} \gg \sqrt{\frac{x}{\log x}} \Big\| \int\limits_{1}^{\sqrt{x}} \Big|\sum_{m \leq z} h(m)\lambda(m) \Big|^2 \frac{\dif z}{z^{2}} \Big\|_{q}^{1/2} - C\sqrt{\frac{x}{\log x}} .
\end{split}
\end{align}

In particular, for any large fixed constant $V$ and Steinhaus random multiplicative function $h(n)$, we have
\begin{align}
\label{hlambdalower1}
\begin{split}
 \Big\| \sum_{n \leq x} h(n)\lambda(n) \Big\|_{2q} \gg \sqrt{\frac{x}{\log x}} \Bigl( \Big\| \int\limits_{-1/2}^{1/2} |F_f(\tfrac12 + \tfrac{4V}{\log x} + it)|^2 \dif t \Big\|_{q}^{1/2} - \frac{C}{e^V} \Big\| \int\limits_{-1/2}^{1/2} |F_f(\tfrac12 + \tfrac{2V}{\log x} + it)|^2 \dif t \Big\|_{q}^{1/2}  - C \Bigr).
\end{split}
\end{align}
 Similarly,  any Rademacher random multiplicative function $h(n)$, we have
\begin{align}
\label{hlambdalower2}
\begin{split}
\Big\| \sum_{n \leq x} h(n)\lambda(n) \Big\|_{2q}  \gg  \sqrt{\frac{x}{\log x}} & \Bigl( \Big\| \int\limits_{-1/2}^{1/2} |F_f( \tfrac12 + \tfrac{4V}{\log x} + it)|^2 \dif t \Big\|_{q}^{1/2}  \\
& - \frac{C}{e^V} \max_{N \in \mz} \frac{1}{(|N|+1)^{1/8}} \Big\| \int\limits_{N-1/2}^{N+1/2} |F_f(\tfrac12 + \tfrac{2V}{\log x} + it)|^2 \dif t \Big\|_{q}^{1/2}  - C \Bigr) .
\end{split}
\end{align}
\end{prop}
\begin{proof}
Recall that $P(n)$ is the largest prime factor of $n$. We let $\epsilon$ be an auxiliary Rademacher random variable that is independent of everything else.  Using arguments similar to those in \cite[section 2.2]{HNR15}, we get, for $2/3 \leq q \leq 1$,
\begin{align}
\label{Elowerbound}
\begin{split}
\E \Biggl|\sum_{\substack{n \leq x, \\ P(n) > \sqrt{x}}} h(n)\lambda(n) \Biggr|^{2q}  = & \frac{1}{2^{2q}} \E \Biggl|\sum_{\substack{n \leq x, \\ P(n) > \sqrt{x}}} h(n)\lambda(n) + \sum_{\substack{n \leq x, \\ P(n) \leq \sqrt{x}}} h(n)\lambda(n) + \sum_{\substack{n \leq x, \\ P(n) > \sqrt{x}}} h(n)\lambda(n) - \sum_{\substack{n \leq x, \\ P(n) \leq \sqrt{x}}} h(n)\lambda(n) \Biggr|^{2q}  \\
\leq & \E \Biggl|\sum_{\substack{n \leq x, \\ P(n) > \sqrt{x}}} h(n)\lambda(n) + \sum_{\substack{n \leq x, \\ P(n) \leq \sqrt{x}}} h(n)\lambda(n) \Biggr|^{2q} + \E \Biggl|\sum_{\substack{n \leq x, \\ P(n) > \sqrt{x}}} h(n)\lambda(n) - \sum_{\substack{n \leq x, \\ P(n) \leq \sqrt{x}}} h(n)\lambda(n) \Biggr|^{2q}  \\
 = & 2 \E \Biggl|\epsilon \sum_{\substack{n \leq x, \\ P(n) > \sqrt{x}}} h(n)\lambda(n) + \sum_{\substack{n \leq x, \\ P(n) \leq \sqrt{x}}} h(n)\lambda(n) \Biggr|^{2q} = 2 \E|\sum_{n \leq x} h(n)\lambda(n)|^{2q} ,
\end{split}
\end{align}
 where the last estimation above follows by observing that the law of 
\begin{equation} \label{decompnsum}
\epsilon \sum_{\substack{n \leq x, \\ P(n) > \sqrt{x}}} h(n)\lambda(n) = \epsilon \sum_{\sqrt{x} < p \leq x} h(p)\lambda(p) \sum_{m \leq x/p} h(m)\lambda(m),
\end{equation}
conditional on the values $(h(p))_{p \leq \sqrt{x}}$ is the same as the law of $\sum_{\substack{n \leq x, \\ P(n) > \sqrt{x}}} h(n)\lambda(n)$. We recast the bound obtained in \eqref{Elowerbound} as
$$ \Big\| \sum_{n \leq x} h(n)\lambda(n) \Big\|_{2q} \gg \Big\| \sum_{\substack{n \leq x, \\ P(n) > \sqrt{x}}} h(n)\lambda(n) \Big\|_{2q} . $$

Observe that when we decompose the sum over $n$ in \eqref{decompnsum}, the inner sums on the right-hand side are determined by $(h(p))_{p \leq \sqrt{x}}$, which are independent of the values $(h(p))_{\sqrt{x} < p \leq x}$. Therefore, writing $\mathbb{E}^{(\sqrt{x})}$ the conditional expectation with respect to $(h(p))_{p \leq \sqrt{x}}$ and applying \eqref{KTlower}, we get that
\begin{align}
\label{Elowercond}
\begin{split}
  \E \Big|\sum_{\substack{n \leq x, \\ P(n) > \sqrt{x}}} & h(n)\lambda(n) \Big|^{2q} = \E\mathbb{E}^{(\sqrt{x})} \Big|\sum_{\sqrt{x} < p \leq x} h(p)\lambda(p) \sum_{m \leq x/p} h(m)\lambda(m)\Big|^{2q} \\
\gg & \E \Big(\sum_{\sqrt{x} < p \leq x} \lambda^2(p) \Big|\sum_{m \leq x/p} h(m)\lambda(m) \Big|^2 \Big)^{q} \geq  \frac{1}{\log^{q}x} \E \Big(\sum_{\sqrt{x} < p \leq x} \lambda^2(p)\log p \Big|\sum_{m \leq x/p} h(m)\lambda(m) \Big|^2 \Big)^{q} .
\end{split}
\end{align}

  Recall that $X = e^{b_4\sqrt{\log x}/4}$.  Using the inequality $|a + b|^2 \geq (1/4)|a|^2 - \min (|b|^2, |a/2|^2 ) \geq 0$ for any real numbers $a,b$, we get that
\begin{align*}
\sum_{\sqrt{x} < p \leq x} & \lambda^2(p) \log p \Bigl|\sum_{m \leq x/p} h(m)\lambda(m) \Bigr|^2 = \sum_{\sqrt{x} < p \leq x} \lambda^2(p) \frac{\log p}{p} X \int\limits_{p}^{p(1+1/X)} \Bigl|\sum_{m \leq x/p} h(m)\lambda(m) \Bigr|^2 \dif t \\
& \geq \frac{1}{4} \sum_{\sqrt{x} < p \leq x} \lambda^2(p) \frac{\log p}{p}X \int_{p}^{p(1+1/X)} \Bigl|\sum_{m \leq x/t} h(m)\lambda(m) \Bigr|^2 \dif t \\
& \hspace*{2cm} - \sum_{\sqrt{x} < p \leq x} \lambda^2(p) \frac{\log p}{p}X \int\limits_{p}^{p(1+1/X)} \min\Big (\Bigl|\sum_{x/t < m \leq x/p} h(m)\lambda(m) \Bigr|^2, \frac{1}{4} \Bigl|\sum_{m \leq x/t} h(m)\lambda(m) \Bigr|^2 \Big ) \dif t .
\end{align*}

  We deduce from \eqref{Elowercond} and the above that for $2/3 \leq q \leq 1$,
\begin{align}
\label{Elower}
\begin{split}
 \E \Big|\sum_{\substack{n \leq x, \\ P(n) > \sqrt{x}}} h(n)\lambda(n) \Big|^{2q} \gg & \frac{1}{\log^{q}x} \E \Bigl( \frac{1}{4} \sum_{\sqrt{x} < p \leq x} \lambda^2(p) \frac{\log p}{p}X \int\limits_{p}^{p(1+1/X)} \Bigl|\sum_{m \leq x/t} h(m)\lambda(m) \Bigr|^2 \dif t \Biggr)^{q} \\
& \hspace*{1cm} - \frac{1}{\log^{q}x} \E\Bigl( \sum_{\sqrt{x} < p \leq x} \lambda^2(p) \frac{\log p}{p}X \int\limits_{p}^{p(1+1/X)} \Bigl|\sum_{x/t < m \leq x/p} h(m)\lambda(m) \Bigr|^2 \dif t \Bigr)^{q}.
\end{split}
\end{align}

Now, H\"{o}lder's inequality leads to
\begin{align}
\label{EPmlarge}
\begin{split}
\E\Bigl( \sum_{\sqrt{x} < p \leq x} & \lambda^2(p) \frac{\log p}{p}X \int\limits_{p}^{p(1+1/X)} \Bigl|\sum_{x/t < m \leq x/p} h(m)\lambda(m) \Bigr|^2 \dif t \Bigr)^q \\
 \leq & \Big(\sum_{\sqrt{x} < p \leq x} \lambda^2(p)\frac{\log p}{p}X \int\limits_{p}^{p(1+1/X)} \E \Bigl|\sum_{x/t < m \leq x/p} h(m)\lambda(m) \Bigr|^2 \dif t \Big)^q \\
 = & \Big(\sum_{\sqrt{x} < p \leq x} \lambda^2(p)\frac{\log p}{p}X \int\limits_{p}^{p(1+1/X)} \sum_{x/t < m \leq x/p}\lambda^2(m) \dif t \Big)^q \\
\ll & \Big( \sum_{\sqrt{x} < p \leq x} \lambda^2(p)\frac{\log p}{p}X \int\limits_{p}^{p(1+1/X)} \frac xp-\frac xt+\Big(\frac x{p}\Big)^{3/5}  \dif t \Big)^q ,
\end{split}
\end{align}
 where the last inequality above follows from \eqref{lambdasquareasymp}.  If $p \leq t \leq p(1+1/X)$, then $\frac xp-\frac xt \leq \frac xp-\frac x{p(1+1/X)} \ll x/(pX)$. It follows that the last expression in \eqref{EPmlarge} is
\begin{align}
\label{EPmlarge1}
\begin{split}
% & \Big( \sum_{\sqrt{x} < p \leq x} \lambda^2(p)\frac{\log p}{p}X \int\limits_{p}^{p(1+1/X)} \frac xp-\frac xt+\left(\frac x{p}\right)^{3/5} \dif t \Big)^q \\
\ll \Big( \sum_{\sqrt{x} < p \leq x} \lambda^2(p)\frac{\log p}{p}X \int\limits_{p}^{p(1+1/X)} \frac{x}{pX}+\Big(\frac x{p}\Big)^{3/5}  \dif t \Big)^q \ll \Big( \sum_{\sqrt{x} < p \leq x} \lambda^2(p)(\log p) \Big(\frac{x}{pX}+\left(\frac x{p}\right)^{3/5} \Big) \Big)^q.
\end{split}
\end{align}

From \eqref{merten3} and partial summation, we get
\begin{align*}
%%\label{EPnsmall}
\begin{split}
\frac{x}{X}\sum_{\sqrt{x} < p \leq x} \lambda^2(p) \frac{\log p}{p} \ll \frac {x\log x}{X} \ll x, \quad \mbox{and} \quad \sum_{\sqrt{x} < p \leq x} \lambda^2(p)\log p \left(\frac x{p}\right)^{3/5} \ll x.
\end{split}
\end{align*}

  We derive from \eqref{EPmlarge}, \eqref{EPmlarge1} and the above that
\begin{align}
\label{EPmlarge2}
\begin{split}
 & \E \Bigl( \sum_{\sqrt{x} < p \leq x} \lambda^2(p) \frac{\log p}{p}X \int\limits_{p}^{p(1+1/X)} \Bigl|\sum_{x/t < m \leq x/p} h(m)\lambda(m) \Bigr|^2 \dif t \Bigr)^q
\ll  x^q.
\end{split}
\end{align}

  It follows from \eqref{Elower} and \eqref{EPmlarge2} that
\begin{align}
\label{Elower1}
\begin{split}
 \E \Big|\sum_{\substack{n \leq x, \\ P(n) > \sqrt{x}}} h(n)\lambda(n) \Big|^{2q} \gg & \frac{1}{\log^{q}x} \E \Bigl(  \sum_{\sqrt{x} < p \leq x} \lambda^2(p) \frac{\log p}{p}X \int\limits_{p}^{p(1+1/X)} \Bigl|\sum_{m \leq x/t} h(m)\lambda(m) \Bigr|^2 \dif t \Bigr)^{q} - \frac{x^q}{\log^{q}x} \\
=& \frac{1}{\log^{q}x}\E \Big( \int\limits_{\sqrt{x}}^{x(1+1/X)} \sum_{t/(1 + 1/X) < p \leq t} \lambda^2(p) \frac{\log p}{p}X \Big|\sum_{m \leq x/t} h(m)\lambda(m) \Big|^2 \dif t \Big)^q - \frac{x^q}{\log^{q}x} \\
\geq & \frac{1}{\log^{q}x}\E\Big( \int\limits_{\sqrt{x}}^{x} \sum_{t/(1 + 1/X) < p \leq t} \lambda^2(p) \frac{\log p}{p}X \Big|\sum_{m \leq x/t} h(m)\lambda(m) \Big|^2 \dif t \Big)^q - \frac{x^q}{\log^{q}x}.
\end{split}
\end{align}

   Note that by \eqref{merten3} and partial summation, we have
\begin{align*}
%%\label{Elower1}
\begin{split}
\sum_{t/(1 + 1/X) < p \leq t} & \lambda^2(p) \frac{\log p}{p}X \\
=& X \Big(\log t-\log \Big(\frac {t}{1 + 1/X}\Big)\Big) +O\Big(X\exp(-b_4\sqrt{\log t/(1 + 1/X)})\Big(1+\Big(\log t-\log \Big(\frac {t}{1 + 1/X}\Big)\Big)\Big) \gg  1.
\end{split}
\end{align*}

  It follows from \eqref{Elower1} and the above that
\begin{align*}
%%\label{Elower2}
\begin{split}
\E|\sum_{\substack{n \leq x, \\ P(n) > \sqrt{x}}} h(n)\lambda(n) |^{2q} \gg & \frac{1}{\log^{q}x} \E \Big( \int\limits_{\sqrt{x}}^{x} \Big|\sum_{m \leq x/t} h(m)\lambda(m) \Big|^2 \dif t \Big)^q- \frac{x^q}{\log^{q}x}   \\
 = & \frac{x^{q}}{\log^{q}x} \E\Big( \int\limits_{1}^{\sqrt{x}} \Big|\sum_{m \leq z} h(m)\lambda(m) \Big|^2 \frac{\dif z}{z^{2}} \Big)^q- \frac{x^q}{\log^{q}x}.
\end{split}
\end{align*}

From this, for some constant $C'>0$, we have
\begin{align}
\label{Elower3}
\begin{split}
\E \Big| \sum_{\substack{n \leq x, \\ P(n) > \sqrt{x}}} h(n)\lambda(n) \Big|^{2q}+ C'\frac{x^q}{\log^{q}x} \gg & \frac{x^{q}}{\log^{q}x} \E \Big( \int\limits_{1}^{\sqrt{x}} \Big|\sum_{m \leq z} h(m)\lambda(m) \Big|^2 \frac{\dif z}{z^{2}} \Big)^q.
\end{split}
\end{align}
 The above implies that
\begin{align}
\label{Elower4}
\begin{split}
\Big (\E\Big|\sum_{\substack{n \leq x, \\ P(n) > \sqrt{x}}} h(n)\lambda(n) \Big|^{2q}+ C'\frac{x^q}{\log^{q}x}\Big )^{1/2q} \gg & \sqrt{\frac{x}{\log x}} \Big\| \int\limits_{1}^{\sqrt{x}} \Big|\sum_{m \leq z} h(m)\lambda(m) \Big|^2 \frac{\dif z}{z^{2}} \Big\|_{q}^{1/2}.
\end{split}
\end{align}
  Note that we have
\begin{align}
\label{Elower5}
\begin{split}
\Big (\E \Big|\sum_{\substack{n \leq x, \\ P(n) > \sqrt{x}}} h(n)\lambda(n) \Big|^{2q}+ C'\frac{x^q}{\log^{q}x}\Big )^{1/2q} \leq & \Big (2\max \Big( \E \Big|\sum_{\substack{n \leq x, \\ P(n) > \sqrt{x}}} h(n)\lambda(n) \Big|^{2q}, C'\frac{x^q}{\log^{q}x}\Big)\Big )^{1/2q}  \\
\leq & 2^{1/2q} \Big\| \sum_{n \leq x} h(n)\lambda(n) \Big\|_{2q}+ (2C')^{1/2q}\sqrt{\frac{x}{\log x}}.
\end{split}
\end{align}
 Now \eqref{hlambdalower} readily follows from \eqref{Elower4} and \eqref{Elower5}. \newline

Ffor \eqref{hlambdalower1}, note first that
\begin{align}
\label{Esmooth}
\begin{split}
\E\left( \int\limits_{1}^{\sqrt{x}} \left|\sum_{m \leq z} h(m)\lambda(m) \right|^2 \frac{\dif z}{z^{2}} \right)^q = \E\left( \int\limits_{1}^{\sqrt{x}} \left|\sum_{\substack{m \leq z, \\ x \text{-smooth}}} h(m)\lambda(m) \right|^2 \frac{\dif z}{z^{2}} \right)^q,
\end{split}
\end{align}
 as $z \leq \sqrt{x}$ in the integrals above. It follows that, for any fixed large constant $V$,
\begin{align*}
%%\label{Elower5}
\begin{split}
 \E\left( \int\limits_{1}^{\sqrt{x}} \left|\sum_{m \leq z} h(m)\lambda(m) \right|^2 \frac{\dif z}{z^{2}} \right)^q \geq \E\left( \int\limits_{1}^{\sqrt{x}} \left|\sum_{\substack{m \leq z, \\ x \text{-smooth}}} h(m)\lambda(m) \right|^2 \frac{\dif z}{z^{2 + 8V/\log x}} \right)^q.
\end{split}
\end{align*}

  As $0<q \leq 1$, we use $(\sum^{\infty}_{i=1}a_i)^q \leq \sum^{\infty}_{i=1}a^q_i$ for convergent series $\sum^{\infty}_{i=1}a_i$ with $a_i \geq 0$ to the last expression above, getting that it is 
\begin{align}
\label{Esmoothdif}
\begin{split}
% & \E\left( \int_{1}^{\sqrt{x}} \left|\sum_{m \leq z} h(m)\lambda(m) \right|^2 \frac{dz}{z^{2}} \right)^q \\
 \geq & \E\left( \int\limits_{1}^{\infty} \left|\sum_{\substack{m \leq z, \\ x \text{-smooth}}} h(m)\lambda(m) \right|^2 \frac{\dif z}{z^{2 + 8V/\log x}} \right)^q - \E\left( \int\limits_{\sqrt{x}}^{\infty} \left|\sum_{\substack{m \leq z, \\ x \text{-smooth}}} h(m)\lambda(m) \right|^2 \frac{\dif z}{z^{2 + 8V/\log x}} \right)^q \\
 \geq & \E\left( \int\limits_{1}^{\infty} \left|\sum_{\substack{m \leq z, \\ x \text{-smooth}}} h(m)\lambda(m) \right|^2 \frac{\dif z}{z^{2 + 8V/\log x}} \right)^q - \frac{1}{e^{2Vq}} \E\left( \int\limits_{1}^{\infty} \left|\sum_{\substack{m \leq z, \\ x \text{-smooth}}} h(m)\lambda(m) \right|^2 \frac{\dif z}{z^{2 + 4V/\log x}} \right)^q .
\end{split}
\end{align}
Now Lemma \ref{parseval} gives that if $x$ is large enough, then for $j=1,2$,
\begin{align}
\label{EParseval}
\begin{split}
 & \int\limits_{1}^{\infty} \left|\sum_{\substack{m \leq z, \\ x \text{-smooth}}} h(m)\lambda(m) \right|^2 \frac{\dif z}{z^{2 + 4jV/\log x}}= \int\limits_{-\infty}^{\infty} \frac{|F_f(\frac 12 + \frac{2jV}{\log x} + it)|^2}{|\frac 12 + \frac{2jV}{\log x} + it|^2} \dif t.
\end{split}
\end{align}

 Note that
\begin{align}
\label{Elargelowerbound}
\begin{split}
 \int\limits_{-\infty}^{\infty} \frac{|F_f(\frac 12 + \frac{2jV}{\log x} + it)|^2}{|\frac 12 + \frac{2jV}{\log x} + it|^2} \dif t =\sum^{\infty}_{N=-\infty} \int\limits_{N-1/2}^{N+1/2}  \frac{|F_f(\frac 12 + \frac{2jV}{\log x} + it)|^2}{|\frac 12 + \frac{2jV}{\log x} + it|^2} \dif t \gg \int\limits_{-1/2}^{1/2} |F_f(\tfrac 12 + \tfrac{2jV}{\log x} + it)|^2 \dif t.
\end{split}
\end{align}

  For the Steinhaus case, the translation invariance that $\int_{N-1/2}^{N+1/2}  |F_f(\frac 12 + \frac{2jV}{\log x} + it)|^2 \dif t=\int_{-1/2}^{1/2}  |F_f(\frac 12 + \frac{2jV}{\log x} + it)|^2 \dif t$ gives
\begin{align}
\label{EsmalllowerS}
\begin{split}
 \int\limits_{-\infty}^{\infty} \frac{|F_f(\frac 12 + \frac{2jV}{\log x} + it)|^2}{|\frac 12 + \frac{2jV}{\log x} + it|^2} \dif t \ll & \sum^{\infty}_{N=-\infty} \frac {1}{(|N|+1)^2} \int\limits_{N-1/2}^{N+1/2}  |F_f(\tfrac 12 + \tfrac{2jV}{\log x} + it)|^2 \dif t \\
= &  \int\limits_{-1/2}^{1/2}  |F_f(\tfrac 12 + \tfrac{2jV}{\log x} + it)|^2 \dif t \sum^{\infty}_{N=-\infty} \frac {1}{(|N|+1)^2} \ll \int\limits_{-1/2}^{1/2}  |F_f(\tfrac 12 + \tfrac{2jV}{\log x} + it)|^2 \dif t.
\end{split}
\end{align}

   For the Rademacher case, we have
\begin{align}
\label{EsmalllowerR}
\begin{split}
 \int\limits_{-\infty}^{\infty} \frac{|F_f(\frac 12 + \frac{2jV}{\log x} + it)|^2}{|\frac 12 + \frac{2jV}{\log x} + it|^2} \dif t \ll & \max_{N \in \mz} \frac{1}{(|N|+1)^{q/4}} \int\limits_{N-1/2}^{N+1/2} |F_f(\tfrac12 + \tfrac{2jV}{\log x} + it)|^2  \dif t \sum^{\infty}_{N=-\infty} \frac {1}{(|N|+1)^{2-q/4}}\\
\ll & \max_{N \in \mz} \frac{1}{(|N|+1)^{q/4}} \int\limits_{N-1/2}^{N+1/2} |F_f(\tfrac12 + \tfrac{2jV}{\log x} + it)|^2  \dif t.
\end{split}
\end{align}

  We deduce from \eqref{Esmooth}--\eqref{EsmalllowerR} that for the Steinhaus case,
\begin{align}
\label{EsmoothestS}
\begin{split}
\E\left( \int\limits_{1}^{\sqrt{x}} \left|\sum_{m \leq z} h(m)\lambda(m) \right|^2 \frac{\dif z}{z^{2}} \right)^q \gg  \E\left( \int\limits_{-1/2}^{1/2}  |F_f(\tfrac 12 + \tfrac{2jV}{\log x} + it)|^2 \dif t \right)^q-e^{-2Vq} \E\left( \int\limits_{-1/2}^{1/2}  |F_f(\tfrac 12 + \tfrac{2jV}{\log x} + it)|^2 \dif t \right)^q.
\end{split}
\end{align}
Ffor the Rademacher case, we get
\begin{align}
\label{EsmoothestR}
\begin{split}
\E\left( \int\limits_{1}^{\sqrt{x}} \left|\sum_{m \leq z} h(m)\lambda(m) \right|^2 \frac{\dif z}{z^{2}} \right)^q \gg  \E &\left( \int\limits_{-1/2}^{1/2}  |F_f(\tfrac 12 + \tfrac{2jV}{\log x} + it)|^2 \dif t \right)^q \\
& -e^{-2Vq} \E\left( \max_{N \in \mz} \frac{1}{(|N|+1)^{q/4}} \int\limits_{N-1/2}^{N+1/2} |F_f(\tfrac12 + \tfrac{2jV}{\log x} + it)|^2 \dif t \right)^q.
\end{split}
\end{align}
 We now deduce from \eqref{hlambdalower}, \eqref{EsmoothestS} and argue as in \eqref{Elower3}--\eqref{Elower5}.  This leads to the bound in \eqref{hlambdalower1}. The estimate in \eqref{hlambdalower2} is proven similarly, using \eqref{EsmoothestR} instead. This completes the proof of the proposition.
\end{proof}

\section{Proof of Theorem \ref{lowerboundsfixedmodmean}}
\label{secmainupper}

\subsection{The upper bounds, Steinhaus}
\label{secmainupper1}

For each $|t| \leq 1/2$ and integers $0 \leq j \leq \log\log x - 2$, we define a sequence $t(j)$ iteratively by setting $t(-1) = t$ and
\begin{align}
\label{tjdef}
\begin{split}
 t(j) := \max\Big\{u \leq t(j-1): u = \frac{n}{((\log x)/e^{j+1}) \log((\log x)/e^{j+1})} \; \text{for some} \; n \in \mz \Big\} .
\end{split}
\end{align}

  It is shown in \cite[Section 4.1]{Harper20} that, for each $j$,
\begin{align}
\label{approxdist}
\begin{split}
|t-t(j)| \leq &  \frac{2}{((\log x)/e^{j+1}) \log((\log x)/e^{j+1})} .
\end{split}
\end{align}

  Let $B$ be the large fixed natural number from Proposition \ref{prop5}, and $\mathcal{G}(k)$ denote the event that for all $|t| \leq 1/2$ and all $k \leq j \leq \log\log x - B - 2$, we have
$$ \Biggl( \frac{\log x}{e^{j+1}} e^{g(x,j)} \Biggr)^{-1} \leq \prod_{l = j}^{\lfloor \log\log x \rfloor - B - 2} |I_{l}(1/2 - \frac{k}{\log x} + it(l))| \leq \frac{\log x}{e^{j+1}} e^{g(x,j)} , $$
where for a large constant $C$,
\begin{align}
\label{gdef}
\begin{split}
g(x,j) :=  C\min\{\sqrt{\log\log x}, \frac{1}{1-q} \} + 2\log\log(\frac{\log x}{e^{j+1}}).
\end{split}
\end{align}

  We recall the definition of $F_{k,f}$ from \eqref{FPdef} for any integer $-1 \leq k \leq \log\log x$. With the above notation, we observe by \eqref{merten1} and \eqref{Eprodsquare} that we have
\begin{align}
\label{EFsquare}
\begin{split}
\E |F_{k,f}(\tfrac12 - \tfrac{k}{\log x})|^2  = & \exp\Big (\sum_{p \leq x^{e^{-(k+1)}}} \frac{\lambda^2(p)}{p^{1 - 2k/\log x}} + O(1) \Big )  \\
= & \exp\Big (\sum_{p \leq x^{e^{-(k+1)}}} \frac{\lambda^2(p)}{p} + O\Big( \sum_{p \leq x^{e^{-(k+1)}}} \frac{k \lambda^2(p) \log p}{p \log x} + 1 \Big) \Big ) \ll \frac{\log x}{e^k}.
\end{split}
\end{align}
We further make use of \eqref{Eprodsquare} and follow the proofs of Key Propositions 1 and 2 in \cite{Harper20} to get the following two results.
\begin{prop}
\label{prop1}
With the natation as above, for all large $x$, and uniformly for $0 \leq k \leq \lfloor \log\log\log x \rfloor$ and $2/3 \leq q \leq 1$, we have
$$ \E\Big( \textbf{1}_{\mathcal{G}(k)} \int\limits_{-1/2}^{1/2} |F_{k,f}(\tfrac12 - \tfrac{k}{\log x} + it)|^2 \dif t \Big)^q \ll \Biggl( \frac{\log x}{e^{k}} C \min\{1, \frac{1}{(1-q)\sqrt{\log\log x}} \} \Biggr)^q.$$
\end{prop}

\begin{prop}
\label{prop2}
With the natation as above. For all large $x$, and uniformly for $0 \leq k \leq \lfloor \log\log\log x \rfloor$ and $2/3 \leq q \leq 1$, we have
$$ \p(\mathcal{G}(k) \; \text{fails}) \ll \exp \Big( -2C \min\Big\{\sqrt{\log\log x}, \frac{1}{1-q} \Big\} \Big). $$
\end{prop}

  Now, by H\"{o}lder's inequality, for any Steinhaus random multiplicative function $h(n)$, 
\begin{align*}
%%\label{Erandfcnupperboundmain2}
\begin{split}
 \E\Big|\sum_{n \leq x} h(n)\lambda(n)\Big|^{2q} \ll & \Big (\E\Big|\sum_{n \leq x} h(n)\lambda(n)\Big|^{2} \Big)^q = \Big( \sum_{n \leq x} \lambda(n)^{2} \Big)^{q} \ll x^q.
\end{split}
\end{align*}  
   where the last estimation above follows from \eqref{lambdasquareasymp}. This implies that for any Steinhaus random multiplicative function $h(n)$, and $1 - 1/\sqrt{\log\log x}<q \leq 1$, we have
\begin{align}
\label{upperSteinhaus}
 \E \Big|\sum_{n \leq x} h(n)\lambda(n)\Big|^{2q} \ll \left( \frac{x}{1 + (1-q)\sqrt{\log\log x}} \right)^q .
\end{align}   
    
We use Proposition \ref{prop1} and Proposition \ref{prop2} by following the proof of the upper bound in \cite[Theorem 1]{Harper20} given in Section 4.1 of \cite{Harper20} upon using \eqref{EupperboundS}.  The bound in \eqref{upperSteinhaus} continues to hold for $2/3 \leq q \leq 1 - 1/\sqrt{\log\log x}$.

\subsection{The upper bounds, Rademacher}
\label{secmainupper2}

  Let $x$ be large, and $0 \leq k \leq \lfloor \log\log\log x \rfloor$ and $2/3 \leq q \leq 1$. We argue as in \cite[Section 4.4]{Harper20} upon using \eqref{EprodsquareR} and the estimation given in \eqref{EFsquare}, arriving at
\begin{align*}
%%\label{Ediff}
\begin{split}
\Big\| \int\limits_{|t| \leq 1/\sqrt{\log\log x}} |F_{k,f}(\tfrac12 - \tfrac{m}{\log x} + it)|^2 \dif t \Big\|_{q} \ll & \frac{\log x}{e^{k} \sqrt{\log\log x}},  \\
\frac{1}{|N|^{1/4}} \Big\| \int\limits_{N-1/2}^{N+1/2} |F_{k,f}(\tfrac{1}{2} - \tfrac{k}{\log x} + it)|^2 \dif t \Big\|_{q}  \ll & \frac{\log x}{e^{k} \sqrt{\log\log x}}, \quad |N| \geq (\log\log x)^2.
\end{split}
\end{align*}

Inserting the above estimations into \eqref{EupperboundR}, it suffices, in order to establish \eqref{upperSteinhaus} for any Rademacher random multiplicative function $h(n)$, to show that uniformly for $0 \leq k \leq \lfloor \log\log\log x \rfloor, 2/3 \leq q \leq 1$ and $|N| \leq (\log\log x)^2$, 
\begin{align}
\label{EtncondRademacher}
\begin{split}
\frac{1}{(|N|+1)^{q/4}} \E \Big( \int\limits_{\substack{|t - N| \leq 1/2, \\ |t| > 1/\sqrt{\log\log x} }} |F_{k,f}(\tfrac{1}{2} - \tfrac{m}{\log x} + it)|^2 \dif t \Big)^q \ll \Biggl( \frac{\log x}{e^{k}} \min \Big\{ 1, \frac{1}{(1-q)\sqrt{\log\log x} } \Big\} \Biggr)^q. 
\end{split}
\end{align}

  Now, to establish \eqref{EtncondRademacher}, we argue as in Section 4.4 of \cite{Harper20}.  Without loss of generality, we may assume that $N=0$ here. We recall the definition of $I_l(s)$ for the Rademacher case given in \eqref{Ildef}. We also recall the definition of the sequence of approximations $(t(j))_{0 \leq j \leq \log\log x - 2}$ in \eqref{tjdef} and the bound on $|t - t(j)|$ in \eqref{approxdist}. We define $D(t) := \lceil \log(1/|t|) \rceil + (B+1)$, where $B$ is as in Proposition \ref{prop6}. We denote $\mathcal{G}^{\text{Rad}}(k,t)$ the event such that for all $k \leq j \leq \log\log x - D - 1$, 
$$ \Biggl( \frac{\log x}{e^{j+1}} e^{g(x,j)} \Biggr)^{-1} \leq \prod_{l = j}^{\lfloor \log\log x \rfloor - D - 1} |I_{l}(\tfrac12 - \tfrac{k}{\log x} + it(l))| \leq \frac{\log x}{e^{j+1}} e^{g(x,j)} , $$
where $g(x,j)$ is defined in \eqref{gdef}. We then write $\mathcal{G}^{\text{Rad}}(k)$ for the event that $\mathcal{G}^{\text{Rad}}(k,t)$ holds for all $1/\sqrt{\log\log x} < |t| \leq 1/2$.  Arguing as in the proof of Key Propositions 3 in \cite{Harper20} by using Proposition \ref{prop6} and \eqref{EFsquare}, we arrive at the following result.
\begin{prop}
\label{prop3-1}
For all large $x$, uniformly for $0 \leq k \leq \lfloor \log\log\log x \rfloor$ and $2/3 \leq q \leq 1$, we have
$$ \E\Big( \textbf{1}_{\mathcal{G}^{\text{Rad}}(k)} \int\limits_{1/\sqrt{\log\log x} < |t| \leq 1/2} |F_{k,f}(\tfrac{1}{2} - \tfrac{k}{\log x} + it)|^2 \dif t \Big)^q \ll \Biggl( \frac{\log x}{e^{k}} C \min\Big\{1, \frac{1}{(1-q)\sqrt{\log\log x}} \Big\} \Biggr)^q , $$
where $\textbf{1}$ denotes the indicator function.
\end{prop}
 
  We note moreover that analogues to \eqref{sumlambdacos}, we have
$$\sum_{x^{e^{-(l+2)}} < p \leq x^{e^{-(l+1)}}} \frac{2\lambda^2(p)\cos(2 t(l) \log p)}{p^{1 - 2k/\log x}} \ll \frac{1}{|t| e^{-l} \log x}.$$ 
  Using the above and proceeding as in the proof of Key Propositions 4 in \cite{Harper20} upon using \eqref{EprodsquareR} and the estimation given in \eqref{EFsquare}, we obtain the following result.
\begin{prop}
\label{prop4}
For all large $x$, uniformly for $0 \leq k \lfloor \log\log\log x \rfloor$ and $2/3 \leq q \leq 1$, we have
$$ \p(\mathcal{G}^{\text{Rad}}(k) \; \text{fails}) \ll \exp \Big( -2C \min\Big\{\sqrt{\log\log x}, \frac{1}{1-q} \Big\} \Big) . $$
\end{prop}

  We now apply Propositions \ref{prop3-1} and \ref{prop4} and use the same arguments as the proof for the  upper bound in the Rademacher case Section 4.4 of \cite{Harper20} to see that the estimation given in \eqref{upperSteinhaus} is also valid for any Rademacher random multiplicative function $h(n)$ and all $2/3 \leq q \leq 1$.

\subsection{The lower bounds, Steinhaus}
\label{secmainlower1}

  Recall that we write $F_f$ for the function $F_{-1,f}$ defined in Section \ref{secupperbounds}. Let $B$ be the large fixed natural number from Proposition \ref{prop5} which we assume without loss of generality is the same as that from Proposition \ref{prop7}. We write $L(t) = L_{x,q,V}(t)$ for each $t \in \mr$ the event that for all $\lfloor \log V \rfloor + 3 \leq j \leq \log\log x - B - 2$,
\begin{eqnarray}
\Biggl(\frac{\log x}{e^{j+1}} \Biggr)^{-B} e^{-\min\{\sqrt{\log\log x}, 1/(1-q)\}} & \leq & \prod_{l = j}^{\lfloor \log\log x \rfloor - B - 2} |I_{l}(\tfrac12 + \tfrac{4V}{\log x} + it)| \nonumber \\
& \leq & \frac{\log x}{e^{j+1}} e^{\min\{\sqrt{\log\log x}, 1/(1-q)\} - 2\log\log(\frac{\log x}{e^{j+1}})}. \nonumber
\end{eqnarray}
We further use $\mathcal{L}$ to denote the random subset of points $|t| \leq 1/2$ at which $L(t)$ occurs. \newline

 We first note the following result which is an analogue to Key Proposition 5 in \cite{Harper20}.
\begin{prop}
\label{keyprop5}
With the notation as above, we have uniformly for all large $x$, $2/3 \leq q \leq 1$ and $1 \leq V \leq (\log x)^{1/100}$, 
$$ \E \Biggl(\int\limits_{\mathcal{L}} |F_f(1/2 + \frac{4V}{\log x} + it)|^2 \dif t \Biggr)^{2} \ll e^{2\min\{\sqrt{\log\log x}, 1/(1-q)\}} \Biggl(\frac{\log x}{V(1 + (1-q)\sqrt{\log\log x})} \Biggr)^2 . $$
\end{prop}
\begin{proof}
  The proof proceeds along the same line as that of Key Proposition 5 in  \cite[p. 70]{Harper20} with modifications to take account into the effects of the extra factors $\lambda(n)$.  More precisely, we need to replace the last display on \cite[p. 70]{Harper20} by 
$$ \exp\Big (\sum_{p \leq x^{e^{-(\lfloor \log\log x \rfloor - B)}}} \frac{\lambda^2(p)(2 + 2\cos(t\log p))}{p^{1+8V/\log x}} + \sum_{x^{e^{-(\max\{\lfloor \log V \rfloor + 3, \lfloor \log(|t|\log x) \rfloor\} + 1)}} < p \leq x} \frac{\lambda^2(p)(2 + 2\cos(t\log p))}{p^{1+8V/\log x}} + O(1) \Big ). $$
  Applying \eqref{merten1}, the above is
$$ \ll \exp\Big (\sum_{\min\{x^{1/V}, e^{1/|t|}\} < p \leq x} \frac{\lambda^2(p)(2 + 2\cos(t\log p))}{p^{1+8V/\log x}}\Big ) \ll \exp\Big (\sum_{\min\{x^{1/V}, e^{1/|t|}\} < p \leq x^{1/V}} \frac{\lambda^2(p)(2 + 2\cos(t\log p))}{p}\Big ). $$
 We then argue as in the proof of Lemma \ref{lem5}.  The overall contribution from those terms with the $\cos(t\log p)$ sum in the above expression is $\ll 1$, so that the above expression is 
 \[ \ll \exp\Big (\sum_{\min\{x^{1/V}, e^{1/|t|}\} < p \leq x^{1/V}} \frac{2\lambda^2(p)}{p}\Big ) \ll \Big(\max\Big\{1, \frac{|t| \log x}{V} \Big\} \Big)^2 \]
  by \eqref{merten1} again. Similarly, we have
\begin{align*}
%%\label{Eest}
\begin{split}
 \E \frac{|F_f(\tfrac12 + \tfrac{4V}{\log x})|^2}{\prod_{l = \lfloor \log(|t|\log x) \rfloor}^{\lfloor \log\log x \rfloor - B - 2} |I_{l}(\tfrac12 + \tfrac{4V}{\log x})|^2} \frac{|F_f( \tfrac12 + \tfrac{4V}{\log x} + it)|^2}{\prod_{l = \lfloor \log(|t|\log x) \rfloor}^{\lfloor \log\log x \rfloor - B - 2} |I_{l}(\tfrac12 + \tfrac{4V}{\log x} + it)|^2} \ll \Big(\frac{|t| \log x}{V} \Big)^2. 
\end{split}
\end{align*}

  Moreover, \eqref{merten1} and \eqref{sttwoprods} yield
\begin{align*}
%%\label{Eest1}
\begin{split}
\E \prod_{\substack{l = \max\{\lfloor \log V \rfloor + 3, \\ \lfloor \log(|t|\log x) \rfloor\}}}^{\lfloor \log\log x \rfloor - B - 2} |I_{l}(\tfrac12 + \tfrac{4V}{\log x})|^2=& \exp\Big (\sum_{p \leq \min\{x^{1/V}, e^{1/|t|}\}} \frac{1}{p^{1 + 8V/\log x}} + O(1) \Big ) \ll \min\Big\{\frac{\log x}{V}, \frac{1}{|t|} \Big\}, \quad \mbox{and} \\
\E \prod_{l = \lfloor \log(|t|\log x) \rfloor}^{\lfloor \log\log x \rfloor - B - 2} |I_{l}(\tfrac12 + \tfrac{4V}{\log x})|^2 =& \exp\Big (\sum_{p \leq e^{1/|t|}} \frac{\lambda^2(p)}{p^{1 + 8V/\log x}} + O(1) \Big ) \ll \frac 1{|t|}.
\end{split}
\end{align*}
  
 We deploy the above estimations into the corresponding places in the proof of Key Proposition 5 in  \cite[p. 70]{Harper20}.  The proposition is follows. 
\end{proof}

  Now, we infer from the proof of the lower bound in Theorem 1 in \cite{Harper20} that
\begin{align}
\label{Elowerbound1}
\begin{split} 
 \E \Bigl( \int\limits_{\mathcal{L}} |F_f( \tfrac12 + \tfrac{4V}{\log x} + it)|^2 \dif t \Bigr)^{q} \geq \frac{\Bigl( \E \int\limits_{\mathcal{L}} |F_f( \tfrac12 + \tfrac{4V}{\log x} + it)|^2 \dif t \Bigr)^{2-q}}{\Bigl( \E \Bigl(\int\limits_{\mathcal{L}} |F_f(\tfrac12 + \tfrac{4V}{\log x} + it)|^2 \dif t \Bigr)^{2} \Bigr)^{1-q}} . 
\end{split}
\end{align}

Now translation invariance implies that 
$$ \Bigl( \int\limits_{-1/2}^{1/2} \E \textbf{1}_{L(t)} |F_f( \tfrac12 + \tfrac{4V}{\log x} + it)|^2 \dif t \Bigr)^{2-q} = \Bigl( \E \textbf{1}_{L(0)} |F_f(\tfrac12 + \tfrac{4V}{\log x})|^2 \Bigr)^{2-q} = \Bigl( \tilde{\p}(L(0))\E |F_f( \tfrac12 + \tfrac{4V}{\log x})|^2 \Bigr)^{2-q}. $$
  Note that by \eqref{merten1}, \eqref{Eprodsquare} and arguing as in \eqref{EFsquare}, we have that 
\begin{align}
\label{EtncondRademacher1}
\begin{split}
  \E |F_f(\tfrac12 + \tfrac{4V}{\log x})|^2=&\exp\Big (\sum_{p \leq x} \frac{\lambda^2(p)}{p^{1 + 8V/\log x}} +O(1)\Big ) = \exp\Big (\sum_{p \leq x^{1/V}} \frac{\lambda^2(p)}{p^{1 + 8V/\log x}} + O(1) \Big ) \gg \frac {\log x}{V}, \quad \mbox{and} \\
  \E |F_{f} \Big( \tfrac12 + \tfrac{2V}{\log x} + it)|^2 =& \exp\Big (\sum_{p \leq x} \frac{\lambda^2(p)}{p^{1 + 4V/\log x}} +O(1)\Big ) = \exp\Big (\sum_{p \leq x^{1/V}} \frac{\lambda^2(p)}{p^{1 + 8V/\log x}} + O(1) \Big ) \ll \frac {\log x}{V}. 
\end{split}
\end{align}
Now apply Proposition \ref{prop5} with $n = \lfloor \log\log x \rfloor - (B+1) - (\lfloor \log V \rfloor + 3), t_j \equiv 0, a = \min\{\sqrt{\log\log x}, 1/(1-q)\} + O(1)$  and $g(n) = -2\log n$ there to estimate $\tilde{\p}(L(0))$. Together with \eqref{EtncondRademacher1}, this gives
\begin{align*}
%%\label{EtncondRademacher1}
\begin{split}
   \Bigl( \int\limits_{-1/2}^{1/2} \E \textbf{1}_{L(t)} |F_f(\tfrac12 + \tfrac{4V}{\log x} + it)|^2 \dif t \Bigr)^{2-q} \gg & \Bigl( \frac{1}{1 + (1-q)\sqrt{\log\log x}} \E |F_f( \tfrac12 + \tfrac{4V}{\log x})|^2 \Bigr)^{2-q} \\
\gg & \Bigl( \frac{\log x}{V(1 + (1-q)\sqrt{\log\log x})} \Bigr)^{2-q} . 
\end{split}
\end{align*}
 
  On the other hand, we apply Proposition \ref{keyprop5} to bound from above the denominator on the right-hand side expression of \eqref{Elowerbound1} to deduce from it that
\begin{align}
\label{Erandfcnlowerboundallq}
  \Big\| \int\limits_{-1/2}^{1/2} |F_f( \tfrac12 + \tfrac{4V}{\log x} + it)|^2 \dif t \Big\|_{q}^{1/2} \gg \sqrt{\frac{\log x}{V(1 + (1-q)\sqrt{\log\log x})}}. 
\end{align}

  Next, we note that as a consequence of Proposition \ref{prop1} and Proposition \ref{prop2}, we deduce by following the proof of the upper bounds in \cite[Theorem 1]{Harper20} given in Section 4.1 of \cite{Harper20} to see that for $2/3 \leq q \leq 1 - 1/\sqrt{\log\log x}$,
\begin{align}
\label{Erandfcnupperbound}
  \E\Big( \int\limits_{-1/2}^{1/2} |F_{f}(\tfrac12 + \tfrac{2V}{\log x} + it)|^2 \dif t \Big)^q \ll \Big( \frac {\log x}{V (1-q) \sqrt{\log\log x}} \Big)^q.
\end{align}
  Moreover, for $1 - 1/\sqrt{\log\log x}<q \leq 1$, we have by H\"{o}lder's inequality and \eqref{EtncondRademacher1} that
\begin{align*}
%%\label{Erandfcnupperboundmain2}
\begin{split}
  \E \Big(\int\limits_{-1/2}^{1/2} |F_{f}(\tfrac12 + \tfrac{2V}{\log x} + it)|^2 \dif t \Big)^q \ll & \Big (\int\limits_{-1/2}^{1/2} \E|F_{f}( \tfrac12 + \tfrac{2V}{\log x} + it)|^2 \dif t \Big)^q \ll \Big( \frac {\log x}{V} \Big)^q.
\end{split}
\end{align*}
  The above and \eqref{Erandfcnupperbound}  implies that, for $2/3 \leq q \leq 1$,
\begin{align}
\label{Erandfcnupperboundallq}
  \Big\| \int\limits_{-1/2}^{1/2} |F_f( \tfrac12 + \tfrac{2V}{\log x} + it)|^2 \dif t \Big\|_{q}^{1/2} \ll \sqrt{\frac{\log x}{V(1 + (1-q)\sqrt{\log\log x})}}. 
\end{align}

  We now substitute \eqref{Erandfcnlowerboundallq} and \eqref{Erandfcnupperboundallq} into \eqref{hlambdalower1}, and choose $V$ large enough to get that for any Steinhaus random multiplicative function $h(n)$, and for any $2/3 \leq q \leq 1$,
\begin{align}
\label{lowerSteinhaus}
 \E\Big| \sum_{n \leq x} h(n)\lambda(n) \Big|^{2q} \gg \left( \frac{x}{1 + (1-q)\sqrt{\log\log x}} \right)^q .
\end{align}   

\subsection{The lower bounds, Rademacher}
\label{secmainlower2}

   We now want to establish \eqref{lowerSteinhaus} for any Rademacher random multiplicative function $h(n)$ as well. As the proof is similar to that for the  Steinhaus case, we only indicate the necessary modifications here. Here we need to have an analogue of Proposition \ref{keyprop5}, which requires a version of Proposition \ref{prop5}. This ultimately relies on altering  \eqref{Eest2} for the Rademacher case, which requires one to compute
$$ \E \prod_{x < p \leq y} \left|1 + \frac{\lambda(p)h(p)}{p^{1/2+\sigma+it_1}}\right|^{2+iu} \left|1 + \frac{\lambda(p)h(p)}{p^{1/2+\sigma+i(t_1 + t_2)}}\right|^{2+iv} , $$
for real $t_1$, $t_2$, $u$, $v$.  One then proceeds as in the proof of Lemma \ref{eulerproduct} to obtain an expression, which is (up to an negligible error term) the exponential of
\begin{align}
\label{EtncondRademacher2}
\begin{split}
  &\sum_{x < p \leq y} \lambda^2(p)\Biggl( \frac{(1+iu/2)^2 + (1+iv/2)^2}{p^{1+2\sigma}} + \frac{(2+iu)(2+iv)\cos(t_1 \log p) \cos((t_1 + t_2)\log p)}{p^{1+2\sigma}} \\
& \hspace*{3cm} + \frac{((2+iu)^{2}/2 - (2+iu))\cos(2t_1 \log p) + ((2+iv)^{2}/2 - (2+iv))\cos(2(t_1 + t_2)\log p)}{2p^{1+2\sigma}} \Biggr) \\
=& \sum_{x < p \leq y} \lambda^2(p)\Biggl( \frac{(1+iu/2)^2 + (1+iv/2)^2}{p^{1+2\sigma}} + \frac{(2+iu)(2+iv)\cos(t_2 \log p)}{2p^{1+2\sigma}}  \\
& + \frac{(2+iu)(iu/2)\cos(2t_1 \log p) + (2+iv)(iv/2)\cos(2(t_1 + t_2)\log p)}{2p^{1+2\sigma}}  + \frac{(2+iu)(2+iv)\cos((2t_1 + t_2)\log p)}{2p^{1+2\sigma}} \Biggr) .  
\end{split}
\end{align}
 Observe that the first two fractions in the last expression above are exactly analogous to the Steinhaus case in \eqref{Eest2}, while the other two are new. \newline

  Recall the definition of $L(t) = L_{x,q,V}(t)$ for each real $t$ from Section \ref{secmainlower1}. We modify the definition $\mathcal{L}$ in Section \ref{secmainlower1} for the Rademacher case to be the random subset of points $t \in [1/3,1/2]$ at which $L(t)$ occurs. This ensures that when $t_1 , t_1 + t_2 \in [1/3,1/2]$, then we have $2/3 \leq 2t_1 , 2(t_1 + t_2), 2t_1 + t_2 \leq 1$. We then argue as in the proof of Lemma \ref{lem5}  to see that the contributions 
$$\sum_{x < p \leq y} \frac{\lambda^2(p)\cos(2t_1 \log p)}{p^{1+2\sigma}} , \quad \sum_{x < p \leq y} \frac{\lambda^2(p)\cos(2(t_1 + t_2) \log p)}{p^{1+2\sigma}} , \quad \mbox{and} \quad \sum_{x < p \leq y} \frac{\lambda^2(p)\cos((2t_1 + t_2) \log p)}{p^{1+2\sigma}}$$
 from the last expression of \eqref{EtncondRademacher2} are all negligible. One then proceeds similar to the proof for the Steinhaus case given in Section \eqref{secmainlower1} to see that \eqref{lowerSteinhaus} holds for any Rademacher random multiplicative function $h(n)$ and for any $2/3 \leq q \leq 1$.

\subsection{Completion of the proof}.
  We now gather our results from Sections \ref{secmainupper1}, \ref{secmainupper2}, \ref{secmainlower1} and \ref{secmainlower2}. The bounds in \eqref{upperSteinhaus} and \eqref{lowerSteinhaus} are valid for both the Steinhaus case and the Rademacher case for all $2/3 \leq q \leq 1$. This establishes \eqref{mainestimation} and hence concludes the proof of Theorem \ref{lowerboundsfixedmodmean}.

\vspace*{.5cm}

\noindent{\bf Acknowledgments.} P. G. is supported in part by NSFC grant 12471003 and L. Z. by the FRG Grant PS71536 at the University of New South Wales.

\bibliography{biblio}
\bibliographystyle{amsxport}

\end{document}